\documentclass[12pt,reqno]{amsart}
\usepackage{hyperref}
\usepackage{amsmath}
 \usepackage{bm,amsthm,amsfonts,latexsym,amssymb,commath} 

\usepackage{graphicx,float,multicol}

\newtheorem{lemma}{\bf{Lemma} }[section]
\newtheorem{proposition}{\bf{Proposition}}[section]
\newtheorem{theorem}{\bf{Theorem}}[section]
\newtheorem{remark}{\sc{Remark} }[section]
\newtheorem{definition}{\sc{Definition} }[section]

\usepackage{siunitx}

\begin{document}

\title{On the wellposedness for a fuel cell problem}

\author{Luisa Consiglieri}
\address{Luisa Consiglieri, Independent Researcher Professor, European Union}
\urladdr{\href{http://sites.google.com/site/luisaconsiglieri}{http://sites.google.com/site/luisaconsiglieri}}

\begin{abstract} 
This paper investigates the existence  of weak solutions to two problems set of elliptic equations in adjoining domains,
 with Beavers--Joseph--Saffman and regularized Butler--Volmer boundary conditions being prescribed on the common interfaces, porous-fluid and membrane, respectively.
Mathematically, the modeling tool  is the coupled Stokes/Darcy problem,  which consists of  the Stokes equation
 on one part of the domain coupled to the Darcy equation, where the flow
velocities are small and mainly driven by the pressure gradient in porous medium,
completed by the thermoelectrochemical  (TEC) system, which consists of 
the energy equation and the mass transport associated with electrochemical reactions, where the fluxes are given by generalized Fourier, Fick and Ohm laws,
 by including the Dufour--Soret and Peltier--Seebeck cross effects, in the multidimensional domain.
The present model includes macrohomogeneous models for both hydrogen and methanol crossover.
The novelty in the presented model lies in the presence of the Joule effect into the Stokes/Darcy-TEC system 
altogether to the quasilinear character given by temperature dependence of
the physical parameters such as the viscosities and  the diffusion coefficients,
by the concentration-temperature dependence of cross-effects coefficients,
and by  the pressure dependence of the permeability. 
The purpose of the present work is to derive quantitative estimates for solutions to
explicit smallness conditions on the data.
 We use  fixed point and compactness arguments based on the quantitative  estimates of approximated solutions.
\end{abstract}


\maketitle

\section{Introduction}
\label{int} 

In this paper, we study the so-called  fuel cells.
Our concern is on the mathematical analysis of thermoelectrochemical (TEC) models  from
 devices that convert the chemical energy from a fuel  into electricity. The chemical reaction produces charged  ions, which move on a membrane. 
Positively charged ions are conducted by the proton exchange membrane 
at low operating  temperature in polymer electrolyte membrane fuel cells (PEMFC) \cite{li2004,fuller,gott}
 and direct methanol fuel cells (DMFC) \cite{sund,wawa},
while negatively charged ions are conducted by adequate ionic condutors  at  high operating temperatures in
 solid oxide fuel cells  (SOFC) \cite{daud,fu}
and molten carbonate fuel cells (MCFC) \cite{Duijn}.
Although their differences, they all consist of two electrodes (one anode and one cathode),
an electrolyte membrane separator between the two electrodes, and  two or more channels.
  The domain consists of different pairwise disjoint  Lipschitz subdomains
(precisely, it is separated into five regions, and it has four interfaces of \((n-1)-\)dimension) as
sagittally illustrated in Fig. \ref{fpem}.
It includes the membrane medium  in contrast with the phase change models, whose consider the membrane as an interface separating the different conductive phases \cite{parma}.
We refer to  \cite{donald}  a numerical approach for the steady free boundary  value problem, on which mixed Dirichlet--Neumann conditions are specified,
motivated by the two-phase flow in fuel cell electrodes.

 It is widely recognized that the behavior of the components in
 the interface boundaries plays an essential role in the cell performance, and it can determine its life.
An exact solution of an electro-osmotic flow problem modeling  polymer electrolyte membranes
is derived in \cite{berg} under the assumption that Stokes flow is driven by only an external field,
not by a pressure gradient, in  an infinite cylindrical pore.
 In \cite{alam,shaw}, the authors combine a PEM fuel cell and electrical circuits
 to potentiate the  energy efficiency, to reduce the cost of FC technology, and to improve fuel usage.
Computational fluid dynamics (CFD) based tools are developed for PEMFC
(see \cite{djilali,cole,umwachen}, and the references therein).

The complexity of the present model is a true drawback
by the presence of both the cross-effects and the Joule effects altogether with different types of interfaces:
the fluid-porous interfaces that require stress boundary conditions such as the   Beavers--Joseph--Saffman interface condition,
and the membrane interface on that the  electrochemical reactions occur.
We refer to \cite{chern} the existence of a weak solution of a  1D half-cell model.

The coupled system of partial differential equations (Stokes/Darcy-TEC) is
quasilinear since the physical parameters such as the viscosity and the diffusion coefficients depend on the temperature
while the cross effects coefficients depend on the temperature and the concentrations.
Moreover,  the permeability depends on the pressure by the Klinkenberg equation. 
 It is known that regularity results are available whenever the Navier--Stokes--Fourier system has
constant coefficients \cite{bveiga}. We refer to \cite{cms2012} the study of the Beavers--Joseph--Saffman--Stokes--Darcy--Fourier problem.

The existence of weak solutions is established by applying a fixed point procedure  under some assumptions on the nonlinear terms.
 The use of the Tychonoff fixed point theorem  is being  somewhat standard.
 However, the existence  of the dissipation term requires some additional regularity  
and some small coefficient conditions were enforced.

 We confine ourselves  in the study of the proton exchange membrane  fuel cell.
 We focus our attention on H\(_2\)PEM fuel cells driven by gaseous hydrogen, but the present model
 may be include other cells as for instance direct methanol fuel cells operating on methanol in an aqueous solution.

The structure of the paper is as follows. We begin by introducing the concrete physical model under consideration.
 Next, the functional framework,
the data under consideration and  the main theorems are stated  in Section \ref{framework}.
Some auxiliary results are proved in Section  \ref{tdt}. In particular, the existence of an auxiliary velocity-pressure pair in
Subsection \ref{auxup}, and an auxiliary partial density-temperature-potential  triple solution in Subsection \ref{auxdi}.
In section \ref{fpthm},  the fixed point argument is applied to prove Theorem \ref{tmain}.

\section{Statement of the fuel cell problem}
\label{spemfc}

Let \(\Omega\) be a bounded multiregion domain   of \(\mathbb{R}^n\), \(n\geq 2\), that is,
\(\Omega=\mathrm{int}\left(\overline{\Omega}_\mathrm{f} \cup \overline{\Omega}_\mathrm{p}\right)\) is  a connected open set,
 with \({\Omega}_\mathrm{f} \) and \({\Omega}_\mathrm{p} \) being two disjoint open subsets of \(\Omega\).
The multidomain \(\Omega\) represents one single PEM fuel cell,
 which its 2D (two-dimensional) representations are schematically illustrated in Fig. \ref{fpem}.
\begin{figure}
\begin{multicols}{2}
\centering 
 \includegraphics[width=0.5\textwidth]{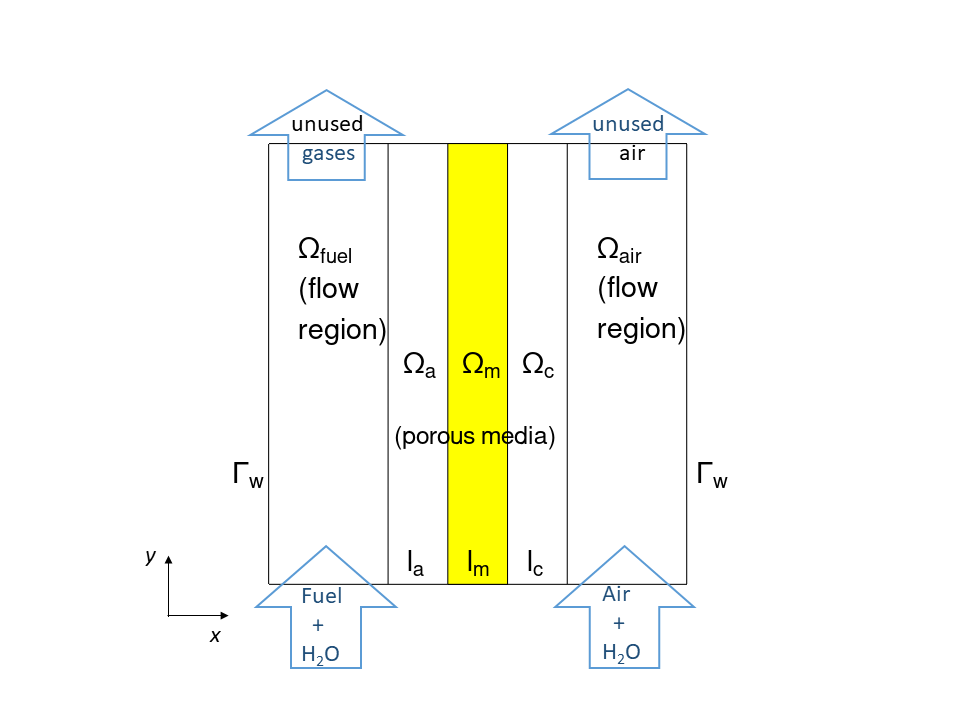} \\
 \includegraphics[width=0.5\textwidth]{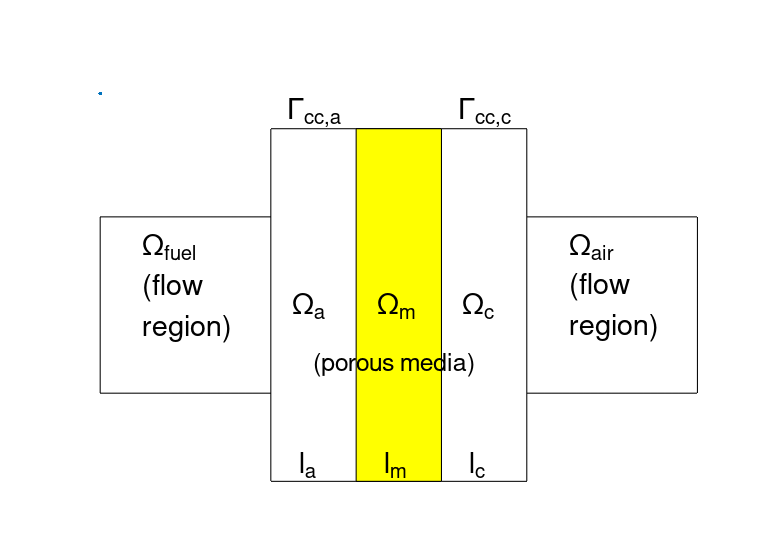}
\end{multicols}
\caption{The flow region \(\Omega_\mathrm{f}= \Omega_\mathrm{fuel}\cup \Omega_\mathrm{air}\) and
 the porous region \(\Omega_\mathrm{p}
= \Omega_\mathrm{a}\cup\overline{\Omega}_\mathrm{m}\cup\Omega_\mathrm{c}\)
(not in scale), with length \(l_\mathrm{a}+l_\mathrm{m}+l_\mathrm{c}<< L\) 
where \(L=1-10\,\si{\centi\metre}\) denotes each channel length. Left: \(xy\) cross-section.  Right: \(xz\) cross-section.}  \label{fpem}
\end{figure}

The fluid bidomain  \({\Omega}_\mathrm{f} \) consists of two channels,
namely the anodic fuel channel \(\Omega_\mathrm{fuel}\) and the cathodic air channel \(\Omega_\mathrm{air}\),
 constituted by  mixtures (since their noncontinuity) of the gas and liquid phases \cite{li2004}.

The membrane electrode assembly, what we call by porous domain  \({\Omega}_\mathrm{p} \), consists of the regions relative to
 the membrane separator \(\Omega_\mathrm{m}\) and  the backing and catalyst layers of the two electrodes.
The domain \(\Omega_\mathrm{m}\)  stands for the proton conducting membrane 
(\SIrange{20}{100}{\micro\metre} in thickness). It  accounts for the transport of
dissolved water (H\(_2\)O) and  the hydronium (H\(_3\)O\(^+\)) ions,
 and it is electrically insulating such that
 the electrons are forced to travel in an external circuit   from the anode to the cathode.
An usual catalyst layer,  between the membrane separator and the backing layer,
 can be assumed to have negligible measure
  (the backing layers are approximately \(l_a=l_c=\SI{200}{\micro\metre}\)  in thickness,
 while the catalyst layers are \SIrange{5}{10}{\micro\metre} \cite{li2004,li2009}),
  and it is denoted by \(\Gamma_\mathrm{CL}\).
Other interface is
the porous-fluid boundary \(\Gamma=\partial\Omega_\mathrm{p}\cap\Omega\).

The fuel (for instance, pure hydrogen \cite{norsk}
 or hydrocarbon type which includes diesel, methanol \cite{wawa} and chemical  hydrides) is oxidized
at the anode catalyst layer \(\Gamma_\mathrm{a}\),   
generating positively charged  ions and electrons.
 The positively charged  ions  travel through   \({\Omega}_\mathrm{m} \), while
 the traveling of the free electrons  produces the electric current in the backing layers, \(\Omega_\mathrm{a}\) 
 and \(\Omega_\mathrm{c}\), through an external circuit that it is attained by a
 current collector \(\Gamma_\mathrm{cc}\).
These two currents are interconnected through the electrochemical reactions.
At the cathode  catalyst layer \(\Gamma_\mathrm{c}\),
 the oxygen reduction occurs: hydrogen ions, electrons, and oxygen react to form water. We set
\[  \Omega_\mathrm{p}=\Omega_\mathrm{a}\cup\left(\overline{\Omega}_\mathrm{m}\cap\Omega\right)\cup\Omega_\mathrm{c}
  =\Omega_\mathrm{a}\cup\Gamma_\mathrm{a}\cup\Omega_\mathrm{m}\cup\Gamma_\mathrm{c}\cup\Omega_\mathrm{c}.
  \] 
Hereafter,  the subscripts, a and c, stand for anode and  cathode, respectively.

The general phenomenological fluxes, \(\mathbf{q}\) [\si{\watt\per \square\metre}],
  \(\mathbf{j}_{i}\)  [\si{\kilogram\per\second\per \square\metre}]
and \(\mathbf{j}\) [\si{\ampere\per \square\metre}],
  are explicitly driven by gradients of  the temperature  \(\theta\),
 the mass concentration vector \(\bm{\rho}\), and  the electric potential \(\phi\),
in the form (up to some temperature and concentration dependent factors)
\begin{align}\label{defji}
\mathbf{j}_{i} &= - D_{i} (\theta) \nabla \rho_{i} -
 \sum_{\genfrac{}{}{0pt}{2}{j=1}{j\not= {i}}}^\mathrm{I} D_{{i}j} (\theta) \nabla \rho_j
-\rho_{i}S_{i}( c_{i} ,\theta) \nabla\theta
-u_{i} \rho_{i}\nabla\phi; \\ \label{defq}
\mathbf{q} &= -R\theta^2 \sum_{j=1}^\mathrm{I} D'_j ( c_j, \theta) \nabla c_j -k (\theta) \nabla\theta
-\Pi (\theta) \sigma (\mathbf{c},\theta) \nabla\phi;\\
\mathbf{j} &= -F\sum_{j=1}^\mathrm{I} z_{j}D_{j}(\theta) \nabla c_j -\alpha_\mathrm{S} (\theta) \sigma (\mathbf{c},\theta) \nabla\theta
-\sigma(\mathbf{c},\theta) \nabla\phi, \nonumber
\end{align}
with \(i=1,\cdots,\mathrm{I} \), see \cite{lap2017,mex} and the references therein.
These include
the Fick law (with the diffusion coefficient \(D_i\) [\si{\metre\squared\per\second}]),
  the Fourier law (with  the thermal conductivity  \(k\)  [\si{\watt\per\metre\per\kelvin}]),
 the Ohm law  (with  the electrical conductivity  \(\sigma\)  [\si{\siemens\per\metre}]),
the Dufour--Soret  cross effect (with the Dufour coefficient \(D'_i\)  [\si{\metre\squared\per\second\per\kelvin}]
 and the Soret coefficient \(S_i\)   [\si{\metre\squared\per\second\per\kelvin}]), and
the Peltier--Seebeck cross effect (with the Peltier coefficient \(\Pi\) [\si{\volt}] and the Seebeck coefficient 
 \(\alpha_\mathrm{S}\)  [\si{\volt\per\kelvin}]  being correlated by the first Kelvin relation).

In the fuel cell model, the main contribution for the electric potential is given at the membrane interface (cf.  Subsection \ref{interface}),
and the electric flux is reduced to the Ohm law \cite{li2004,li2009}
\begin{equation}\label{ohm}
\mathbf{j} = -\sigma(\mathbf{c},\theta) \nabla\phi \quad\mbox{  in }\Omega_\mathrm{a}\cup\Omega_\mathrm{c}.
\end{equation}

Each  partial density is defined by
\begin{equation}
\rho_i = M_ic_i,
\end{equation}
where 
\( M_i\) denotes the molar mass [\si{\kilogram\per\mole}] and 
\(c_i\) is the  molar concentration [\si{\mole\per\cubic\metre}] of the species  \(i\).

 The mobility \(u_i\)  [\si{\metre\squared\per\second\per\volt}] satisfies the Nernst-Einstein relation \(u_i= z_i F D_i/(R\theta)\), and according to
Onsager reciprocal theorem, the two coupling coefficients are equal.
The universal constants are the so-called the Faraday constant
 \(F= \SI{9.6485e4}{\coulomb\per\mole}\), and the gas constant
 \(R= \SI{8.314}{\joule\per\mole\per\kelvin}\). 

 Hereafter the subscript \(i\) stands for the correspondence to the ionic component \(i=1,\cdots,\mathrm{I}\) 
intervened in the reaction process, with \(\mathrm{I}\in\mathbb{N}\) being either
 \(\mathrm{I}_\mathrm{p}\) whenever \(\Omega_\mathrm{p}\) or  \(\mathrm{I}_\mathrm{f}\) whenever \(\Omega_\mathrm{f}\).
To avoid confusing, we never show the components of vectors of \(\mathbb{R}^n\),  namely
the velocity vector or the gradient, in the present work.

\subsection{In the fluid bidomain  \({\Omega}_\mathrm{f}= \Omega_\mathrm{fuel}\cup \Omega_\mathrm{air} \)}

By the characteristics of the channels, the convection for fluid and heat flows may be neglected.

The governing equations are the conservation of mass, momentum,  species
and energy, a.e. in \(\Omega_\mathrm{f}\),
\begin{align}
\nabla\cdot(\rho \mathbf{u})=0; \label{mass}\\
-\nabla  \cdot\tau =-\nabla p ; \label{momentum}\\
\nabla\cdot(\mathbf{u} \rho_i)+\nabla\cdot \mathbf{j}_i=0; \label{species}\\
\label{heateqs}
\nabla\cdot \mathbf{q} =0,
\end{align}
for the uncharged species \(i=1,\cdots,\mathrm{I}\).
 The unknown functions are the density \(\rho\), the velocity \(\mathbf{u}=(u_x,u_y,u_z)\),
 the mass concentration vector \(\bm{\rho}=(\rho_1,\cdots, \rho_\mathrm{I})\) and the temperature \(\theta\).

We assume that the anode and cathode gas mixtures with water vapor act as ideal gases \cite{gott}, that is,
the pressure \(p\) obeys the Boyle--Marriotte law
\begin{equation}\label{boyle}
p = R_\mathrm{specific}\rho\theta,
\end{equation} 
where \( R_\mathrm{specific}=R/M\) with \( M\) denoting the molar mass [\si{\kilogram\per\mole}].
Moreover, the deviatoric stress tensor \(\tau=pI+\sigma\),
 where \(\sigma\) represents the Cauchy stress tensor and \( \mathsf{I}\) denotes the identity (\(n\times n\))-matrix.
The stress tensor \( \tau\), which is temperature dependent, obeys the constitutive law 
\begin{equation}\label{diff}
\tau =\mu(\theta)D\mathbf{u}+\lambda(\theta) \mathrm{tr}(D\mathbf{u}) \mathsf{I},
\qquad \mathrm{tr}(D\mathbf{u})=\mathsf{I}:D\mathbf{u}= \nabla\cdot\mathbf{u},
\end{equation}
where \( D=(\nabla+\nabla^T)/2\) denotes the symmetric gradient, and
 \( \mu\) and \( \lambda\) are the viscosity coefficients  in 
accordance with the second law of thermodynamics
\[
\mu(\theta)>0,\quad \nu(\theta) :=n\lambda(\theta)+\mu(\theta)\geq 0,
\] 
with \(\nu\) denoting the bulk (or volume) viscosity  and \(\mu/2\) being the shear (or dynamic) viscosity.
 Here we denote \(\zeta:\varsigma=\zeta_{ij}\varsigma_{ij}\)  taking into account the convention on
implicit summation over repeated indices.

Finally, we emphasize that there is no electric current in the fluid bidomain.

\subsection{Number of species \(\mathrm{I}\)}

 The  number of  species \(\mathrm{I}\in\mathbb{N}\) may indeed represent different numbers
 \(\mathrm{I}_\mathrm{a}\),
 \(\mathrm{I}_\mathrm{m}\) and  \(\mathrm{I}_\mathrm{c}\) corresponding to the domains
 \(\Omega_\mathrm{fuel}\cup\Omega_\mathrm{a}\), \(\Omega_\mathrm{m}\) and 
\(\Omega_\mathrm{air}\cup\Omega_\mathrm{c}\), respectively.

The flow domain \({\Omega}_\mathrm{f} \) accounts for the reactant gases 
(oxygen, nitrogen and water vapor on the cathode channel) and liquid water. 
In the  H\(_2\)PEMFC,  
 the dry hydrogen gas is humidified before of being introduced into the fuel channel  \(\Omega_\mathrm{fuel}\).
In the DMFC, the chemical reaction in the anode catalyst
layer is the methanol (CH\(_3\)OH) oxidation.

\subsubsection{The anodic fuel compartment \(\Omega_\mathrm{fuel}\cup\Omega_\mathrm{a}\)}
\label{anodic} 

 In the anodic region, 
different kinds of humidified fuel may be considered,  for instance:
 \begin{itemize}
 \item  H\(_2\), to produce the hydrogen oxidation reaction \cite{li2004,li2009,fuller}:
 H\(_2\) \(\rightarrow 2\)H\(^+ + 2\mathrm{e}^-\)  on  the membrane interface \(\Gamma_\mathrm{a}\), which means
\[
\mathrm{H}_2(g) +2\mathrm{H}_2\mathrm{O} (l) \rightarrow 2\mathrm{H\mathrm}_3\mathrm{O}^+(aq) + 2\mathrm{e}^-, \qquad
E^0= \SI{0.00}{\volt} . 
\]
The gas composition obeys 
\begin{align*}
\rho &= M(\mathrm{H}_2)c_{\mathrm{H}_2} + M(\mathrm{H}_2\mathrm{O})c_{\mathrm{H}_2\mathrm{O}}
\quad\mbox{ in } \Omega_\mathrm{fuel}\cup\Omega_\mathrm{a} ; \\
\rho &= Mc_{\mathrm{H}_3\mathrm{O}^+} + M(\mathrm{H}_2\mathrm{O})c_{\mathrm{H}_2\mathrm{O}} \quad\mbox{ in } \Omega_\mathrm{m} ,
\end{align*}
with  \(M(\mathrm{H}_2)= \SI{2}{\gram\per\mole}\), \(M(\mathrm{H}_2\mathrm{O})= \SI{18}{\gram\per\mole}\)
 and \(M=\SI{1}{\gram\per\mole}\).

\item methanol, to produce the oxidation reaction \cite{wawa}:
CH\(_3\)OH+H\(_2\)O \(\rightarrow \)  CO\(_2+6\)H\(^+ + 6\)e\(^-\)  
on  the membrane interface \(\Gamma_\mathrm{a}\).
 \end{itemize}

Then, we take \(i=\)  fuel, H\(_2\)O.

 \subsubsection{The cathodic air compartment \(\Omega_\mathrm{air}\cup\Omega_\mathrm{c}\)}
\label{cathodic}

In the cathodic region, the air undergoes the oxygen reduction reaction \cite{fuller,norsk}:
 O\(_2+4\)H\(^+ + 4\mathrm{e}^- \rightarrow 2\)H\(_2\)O\(\)
 on the membrane interface \(\Gamma_\mathrm{c}\), which means
\[
\mathrm{O}_2(g) +4\mathrm{H}_3\mathrm{O}^+ (aq)+ 4\mathrm{e}^- \rightarrow 6\mathrm{H}_2\mathrm{O}(l), \qquad
E^0=\SI{1.23}{\volt}.
\]
Then,  we take \(i=\)  O\(_2\), H\(_2\)O.
The liquid water byproduct drains away for a proper operating of the fuel cell.
The gas composition obeys
\[
\rho = M(\mathrm{H}_2\mathrm{O})c_{\mathrm{H}_2\mathrm{O}}+M(\mathrm{O}_2)c_{\mathrm{O}_2},
\]
with \(M(\mathrm{H}_2\mathrm{O})=\SI{18}{\gram\per\mole} \) and
 \(M(\mathrm{O}_2)=\SI{32}{\gram\per\mole}\).
The mass density is assumed to be
\begin{equation}\label{massdensity}
\rho =\sum_{i=1}^\mathrm{I}\rho_i\mbox{ in }\Omega_\mathrm{f}. 
\end{equation}

Therefore, the overall balanced cell reactions are
\begin{description}
\item[H\(_2\)PEMFC] 
\(
2\mathrm{H}_2+ \mathrm{O}_2\rightarrow 2\mathrm{H}_2\mathrm{O}\), \(
E^0_\mathrm{cell}=\SI{1.23}{\volt}.
\)

\item[DMFC]
\(
2\mathrm{CH}_4\mathrm{O}+3\mathrm{O}_2\rightarrow 2\mathrm{CO}_2+4\mathrm{H}_2
\mathrm{O}.
\)
\end{description}
For the sake of simplicity,  we consider the number of species (cf. Table 1)
\[\mathrm{I} = \mathrm{I}_\mathrm{a} =\mathrm{I}_\mathrm{m} =\mathrm{I}_\mathrm{c} =2.\]
 \begin{table}[h]\label{table1}
\caption{The correspondence of each component to each region}
\center \begin{tabular}{|c|c|c|c|}
\hline  \(i\)&  \(\Omega_\mathrm{fuel}\cup\Omega_\mathrm{a}\) & \(\Omega_\mathrm{m}\)  &  \(\Omega_\mathrm{air}\cup\Omega_\mathrm{c}\)    
 \\
\hline 
1 & fuel & H\(_3\)O\(^+\) &   O\(_2\) \\ 
2 &H\(_2\)O &H\(_2\)O  &H\(_2\)O  \\ 
\hline 
\end{tabular} 
 \end{table}
 The water  is present in  fluid and vapor  states, and in both cases it can be modeled as
a Newtonian fluid (linearly viscous fluid).

\subsection{In the porous domain  \( \Omega_\mathrm{p}=\Omega_\mathrm{a}\cup\Omega_\mathrm{m}\cup\Omega_\mathrm{c}\)}

The governing equations,  after a volume averaging procedure \cite{whit}, are
\begin{align}
\nabla  \cdot \mathbf{u}_\mathrm{D} &=0 ; \label{momentump}\\
\nabla\cdot \mathbf{j}_i &=0; \label{speciesp}
 \\ \label{heatp}
\nabla\cdot \mathbf{q}  &=Q, \mbox{ a.e. in }\Omega_\mathrm{p},
\end{align}
for   \(i= \) according to Table 1 in Section \ref{table1}, that is I\(_\mathrm{p}=2\).
Here, it   is omitted the bracket \(\langle\cdot  \rangle\),  which usually represents  the volume averaged. Thus,
 the temperature \(\theta\) is the spatially averaged (over a representative elementary volume) microscopic quantity, 
and the Darcy velocity \(\mathbf{u}_\mathrm{D}\) [\si{\metre\per\second}] is the superficial average quantity.
 The volume averaged density  \(\rho\) 
 of the fluid is piecewise  constant, \(\rho_\mathrm{water}=\SI{970}{\kilogram\per\cubic\metre}\)
and \(\rho_\mathrm{air}=\SI{0.995}{\kilogram\per\cubic\metre}\),  due to \(\rho_\mathrm{air} =p_\mathrm{atm}M_\mathrm{air}/(R\theta_r)\), 
at the typical operating temperature of \(\theta_r = \SI{357.15}{\kelvin}\) (= \SI{84}{\celsius}),
\(p_\mathrm{atm} =\SI{101.325}{\kilo\pascal}\) and
\(M_\mathrm{air} =\SI{28.97}{\gram\per\mole}\).

The  Darcy velocity \(\mathbf{u}_\mathrm{D}\) obeys
\begin{equation}\label{darcy1}
\mu\mathbf{u}_\mathrm{D}=- K_g \nabla p ,
\end{equation}
where \(p\)  is the intrinsic average  pressure [\si{\pascal}],
 \(\mu=\mu(\theta)\) denotes the viscosity [\si{\pascal\second}]  and
 \(K_g\) represents the gas permeability [\si{\metre\squared}] that is given by the 
 Klinkenberg equation
 \begin{equation}\label{defKg}
 K_g=K_l\left(1+\frac{b}{p}\right),
 \end{equation}
 with \(b>0\) being a constant, \(b>0\) in \(\Omega_\mathrm{a}\cup\Omega_\mathrm{c}\) and  \(b=0\) in \(\Omega_\mathrm{m}\),
and \(K_l>0\) being the liquid permeability of the porous media  that only depends on the porosity \(\epsilon\)
and therefore it is constant.

The molar flux \(\mathbf{J}_i\) of the water \(i=\) H\(_2\)O obeys \eqref{defji},
 where the second term  means the electro-osmosis (\(j\not= i\)), 
 with  \(D_{ij}=n_\mathrm{d}\) representing the electro-osmostic drag coefficient \cite{cole}.
The proton flux \(\mathbf{J}_i\) of the ionic component \(i=\) H\(_3\)O\(^+\) obeys \eqref{defji},
where in the first term \(D_i=\kappa/(z_iF)\), with
 the proton ionic conductivity $\kappa$ being no constant in accordance with the membrane did not being fully
hydrated.

In the energy equation (\ref{heatp}),  the coefficient  \(c_\mathrm{v}>0\) denotes the specific heat capacity of the fluid at constant volume.
The Joule effect 
\begin{equation}\label{joule}
Q=\chi_{\Omega_\mathrm{a}\cup\Omega_\mathrm{c}} \sigma|\nabla\phi|^2
\end{equation}
takes into account that the effect of flow velocity is negligible when
compared to the electrical current that exists in \(\Omega_\mathrm{a}\cup\Omega_\mathrm{c}\).
 
 The electric current density \(\mathbf{j}\) verifies
\begin{equation}\label{electric}
\nabla\cdot\mathbf{j}=0 \quad\mbox{ a.e. in } \Omega_\mathrm{a}\cup\Omega_\mathrm{c} .
\end{equation}
Notice that there is no  electric current density in \(\Omega_\mathrm{m}\), \textit{i.e.} 
there is the ionic current density \( \mathbf{j}_\mathrm{m}\)  that verifies 
\(\mathbf{j}_\mathrm{m}= z_{\mathrm{H}^+} F \mathbf{J}_{\mathrm{H}^+}\), where
 the valence of species \(z_{\mathrm{H}^+}=1\).
Also, \(\sigma_m=\SI{ 8.3}{\siemens\per\metre}\) is known for the ionomer Nafion.

\subsection{On the outer boundary \(\partial\Omega\)}

The boundary of \(\Omega\)   is constituted by three pairwise disjoint open \((n-1)\)-dimensional sets, namely 
 \(\Gamma_\mathrm{in}\),  \(\Gamma_\mathrm{out}\) and \(\Gamma_\mathrm{w}\)
which represent the inlet, outlet and wall boundaries, respectively, 
\[
\partial\Omega= \Gamma_\mathrm{in}\cup \Gamma_\mathrm{out}\cup  \overline \Gamma_\mathrm{w}.
\]
The wall boundary has a subpart  \(\Gamma_\mathrm{cc}\subset \partial\Omega_\mathrm{p}\)
 that stands for the  current collector, meaning that the remaining  wall boundary is electrical current  insulated.
 The inlet and outlet sets are the union of two disjoint connected open  \((n-1)\)-dimensional sets, namely,
 \begin{align*}
 \Gamma_\mathrm{in} &= \Gamma_\mathrm{in,a}\cup  \Gamma_\mathrm{in,c} ; \\
 \Gamma_\mathrm{out} &= \Gamma_\mathrm{out,a}\cup  \Gamma_\mathrm{out,c},
\end{align*}
corresponding to the anodic and cathodic channels, \(\Omega_\mathrm{fuel}\) and \(\Omega_\mathrm{air}\). 

On the wall boundary \(\Gamma_\mathrm{w} \), 
the no outflow boundary conditions are considered to the velocity  and the species, 
\begin{equation}
\mathbf{u}\cdot\mathbf{n}=(\rho_i \mathbf{u} + \mathbf{j}_i) \cdot\mathbf{n}=0 
\quad (i=1,\cdots,\mathrm{I}).\label{noflow}
\end{equation}
Hereafter,  \(\mathbf{n}\) denotes the outward unit normal to \(\partial\Omega\).

On the inlet and outlet boundaries \(\Gamma_\mathrm{in}\cup\Gamma_\mathrm{out}\),
 the velocity, the partial densities and the temperature are specified.
Due to the characteristics of the  domain, the velocity is constantly specified on the \(y\) direction.
Since the general case for prescribed partial densities and temperature can be handled by subtracting  background
profile that fits the specified  functions, we assume homogeneous Dirichlet condition.
\begin{itemize}
\item for a.e. \((x,0,z)\in \Gamma_\mathrm{in}\):
\begin{align*}
\mathbf{u}(x,0,z) &= u_\mathrm{in}\mathbf{e}_y\equiv (0,u_\mathrm{in},0) ;\\
\rho_i(x,0,z) &= \theta(x,0,z) =0.
\end{align*}
\item for a.e. \((x,L,z)\in \Gamma_\mathrm{out}\):
\begin{align*}
\mathbf{u}(x,L,z) &= u_\mathrm{out}\mathbf{e}_y\equiv (0,u_\mathrm{out},0);\\
\rho_i(x,L,z) &= \theta(x,L,z) =0.
\end{align*}
\end{itemize}

On the current collector wall boundary \(\Gamma_\mathrm{cc}\),
 the electric potential is prescribed through the cell voltage
\(E_\mathrm{cell}=\phi|_{\Gamma_\mathrm{cc,c}} - \phi|_{\Gamma_\mathrm{cc,a}} \), that means
\begin{equation}\label{Ecell}
\phi=E_\mathrm{cell} \mbox{ on } \Gamma_\mathrm{cc,c}\quad\mbox{and}\quad \phi=0
\mbox{ on } \Gamma_\mathrm{cc,a}.
\end{equation}
On the remaining  wall boundary \(\Gamma_\mathrm{w}\setminus\Gamma_\mathrm{cc}\),
 the no outflow \( \mathbf{j} \cdot\mathbf{n}=0\) is considered.

Finally, the Newton law of cooling, which is
mathematically known as  the Robin-type boundary condition, is considered
\begin{equation}\label{newton}
 \mathbf{q}\cdot\mathbf{n} = h_c(\theta-\theta_e) \mbox{ on } \Gamma_\mathrm{w} ,
\end{equation}
where \(h_c\) denotes the conductive heat transfer coefficient, which may 
 depend both on the spatial variable and the temperature function \(\theta\),
 and  \(\theta_e\) denotes the external coolant stream temperature at the wall.

\subsection{On the fluid-porous  interface  \(\Gamma\)}
\label{porousfluidbd}

The unit outward  normal to the interface boundary \(\Gamma\) pointing from the fluid region to the porous
medium is \(\mathbf{e}_x\) on int\(\left(\partial\Omega_\mathrm{fuel}\cap \partial \Omega_\mathrm{a}\right)\)
and \( - \mathbf{e}_x\) on int\(\left(\partial\Omega_\mathrm{air}\cap \partial \Omega_\mathrm{c}\right)\).

We consider the continuity of mass flux, a  constant interface temperature,
 and the balance of normal Cauchy stress vectors (namely,   \(\sigma_{fN}+\sigma_{pN}=0\))
\begin{align}
\mathbf{u}\cdot\mathbf{e}_x &= \mathbf{u}_D\cdot\mathbf{e}_x;\label{cmf}\\
 \theta_f &=\theta_p; \label{ttfp} \\
(\tau \cdot \mathbf{e}_x)\cdot\mathbf{e}_x &= [p] := p_\mathrm{f}-p_\mathrm{p} ,\label{cns}
\end{align}
where \([\cdot]\) denotes the jump of a quantity across the interface in direction to the fluid medium.
 The condition (\ref{cmf}) guarantees that the exchange of fluid between the two domains is conservative. 

The heat transfer transmission  is  completed by  the continuous heat flux condition
\begin{equation}
\label{bcsp}
\mathbf{q} \cdot\mathbf{e}_x= -\mathbf{q}_\mathrm{p}\cdot\mathbf{e}_x .
  \end{equation}

Finally, we assume the fluid flow is almost parallel to the interface
 and the Darcy velocity is much smaller than the slip velocity. Thus,
 the Beavers--Joseph--Saffman (BJS)  interface boundary condition may be considered \cite{er}
\begin{equation}
(\tau \cdot\mathbf{n})\cdot \mathbf{e}_j =-\beta \mathbf{u} \cdot\mathbf{e}_j\qquad (j=y,z) \label{bj}
\end{equation}
where the coefficient \(\beta=\alpha_{BJ} K^{-1/2}>0\) denotes the Beavers--Joseph slip coefficient,
 with \(\alpha_{BJ}\) being dimensionless and characterizing the nature of the porous surface.

\subsection{On the  membrane interface \(\Gamma_\mathrm{CL}=\Gamma_\mathrm{a}\cup\Gamma_\mathrm{c}\)}
\label{interface}

In the sequel, we foccus on the H\(_2\)PEMFC. In both half cell reactions,  the number of electrons that participate in each half cell reaction \(n\) is equal to \(4\)
(see Subsections \ref{anodic} and \ref{cathodic}).

On \(\Gamma_\mathrm{a}=\partial\Omega_\mathrm{a}\cap\overline{\Omega}_\mathrm{m}\),
 it occurs the oxidation reaction of the fuel, that is,
 \[
 \mathbf{j}_1\cdot\mathbf{e}_x= -\frac{sM(\mathrm{H}_2)}{nF} j_{a} \quad \mbox{ a.e. on }\Gamma_\mathrm{a} ,
 \] 
with the anodic stoichiometry number \(s=2\).
 
On \(\Gamma_\mathrm{c}=\partial\Omega_\mathrm{c}\cap\overline{\Omega}_\mathrm{m}\),
 it occurs the oxygen reduction reaction, that is,
 \[ 
 \mathbf{j}_1\cdot\mathbf{e}_x= -\frac{sM(\mathrm{O}_2)}{nF} j_{c} \quad \mbox{ a.e. on }\Gamma_\mathrm{c} ,
 \]
 with  the cathodic stoichiometry number \(s=1\).

The reaction rates \(j_\ell\) [\si{\ampere\per\square\metre}] are given by the Butler--Volmer equation
\begin{align*}
j_\mathrm{a} &= j_{\mathrm{a},0} \left( \frac{c_\mathrm{fuel}}{c_{\mathrm{fuel},0}}\right)^{\nu}
\left( \exp\left[\frac{ F\eta_a}{R\theta_a}\right]-
 \exp\left[- \frac{ F\eta_a}{R\theta_a}\right]\right);\\
j_\mathrm{c} &= j_{\mathrm{c},0}  \frac{c_{\mathrm{O}_2}}{c_{\mathrm{O}_2,0}} 
\left( \exp\left[\frac{ F\eta_c}{R\theta_c}\right]-
 \exp\left[- \frac{ F\eta_c}{R\theta_c}\right]\right)
\end{align*}
for some  \(j_{\ell,0}>0\) only spatial dependent being such that \(j_{\mathrm{a},0}>j_{\mathrm{c},0}\) \cite{gott}. 
Here, it is considered the charge transfer coefficient equal to \(1/2\),  \(\theta_a\) and \(\theta_c\) are some reference temperatures, \(\nu=1/2\) for H\(_2\) fuel,
and \(\eta_\ell =\phi_\ell-\phi_m-\phi_r\) stands for the overpotential (\(\ell = a, c\)), for some reference potential \(\phi_r\).

Thus, the electric current may be modeled by the Butler--Volmer boundary condition
\begin{equation}\label{BVeq}
- \mathbf{j}\cdot\mathbf{e}_x=j_\ell\quad \mbox{ a.e. on }\Gamma_\ell  ,\quad ( \ell = \mathrm{a,c}) ,
\end{equation}
 Notice that the reaction rates are affected by the transport of species near the electrode, and may be represented as a current in terms of the limiting current 
\[j_\ell= j_{\ell,L} \left(1-\left( \frac{c}{c_0}\right)^{\nu_\ell} \right),\]
with \(\nu_\ell = \nu\) if \( \ell =a\),  and \(\nu= 1\) if \(\ell =c\).
Then, we may consider 
\begin{equation}\label{etapo}
j_\ell (\eta)= j_{\ell,L}  \frac{2 j_{\ell,0}  \sinh [ \eta /B_\ell]}{ j_{\ell,L} +2  j_{\ell,0}  \sinh[ \eta/B_\ell]} \quad\mbox{ for }\eta\geq 0,
\end{equation}
with \(B_\ell = R\theta_\ell/F\) being the Tafel slope at  \(\ell= a, c\). For a mathematical analysis, we assume that
\begin{equation}\label{etane}
j_\ell(\eta)=-j_\ell(-\eta)\quad\mbox{ if } \eta<0.
\end{equation} 

We emphasize that this assumption avoids the existence of infinitely many non-trivial solutions that happens for boundary value problem
under the Butler--Volmer boundary condition \cite{vogelius}. Also \(j_{\ell}\) representing the dual-pathway kinetic equation
based on the Tafel--Heyrovsky--Volmer mechanism  \cite{kher,jacsHOR} may be similarly treated.

\section{Variational formulation and main result}
\label{framework}

In the framework of Sobolev and Lebesgue functional spaces, for \(r>1\), we introduce the following spaces of test functions
\begin{align*}
\mathbf{V}(\Omega_f) =& \{ \mathbf{v}\in \mathbf{H}^1(\Omega_\mathrm{f}):\ 
 \mathbf{v}=\mathbf{0}\mbox{ on }\Gamma_\mathrm{in}\cup\Gamma_\mathrm{out}; \
 \mathbf{v}\cdot\mathbf{n}=0\mbox{ on }\Gamma_\mathrm{w} \};\\
V_r(\Omega_p) =& \{ v\in W^{1,r}(\Omega_\mathrm{p}):\ v=0\mbox{ on }\Gamma_\mathrm{cc} \};\\
V(\Omega) =& \{ v\in H(\Omega):\ v=0\mbox{ on }\Gamma_\mathrm{in}\cup\Gamma_\mathrm{out}\};\\
H(\Omega_p) =& \{v\in H^{1}(\Omega_\mathrm{p}):\  v_a :=v|_{\Omega_\mathrm{a}}, \  v_c :=v|_{\Omega_\mathrm{c}},
\  v_m :=v|_{\Omega_\mathrm{m}} \\
& v_a =v_m \mbox{ on }\Gamma_\mathrm{a}, \  v_c =v_m \mbox{ on }\Gamma_\mathrm{c}\} ;\\
H(\Omega) =& \{v\in H^{1}(\Omega):\ v_f:=v|_{\Omega_\mathrm{f}},\  v_p:=v|_{\Omega_\mathrm{p}},\
 v_f=v_p \mbox{ on }\Gamma\} ,
\end{align*}
with their usual norms. 
Considering that the Poincar\'e inequality occurs whenever the trace of the function vanishes on a part with positive measure of the boundary \(\partial\Omega\),
then the  Hilbert spaces, \( \mathbf{V}(\Omega_f)\), \( V(\Omega_p) \) and  \(V(\Omega) \),  are  endowed with the standard seminorms.

We denote \(V(\Omega_p)=V_2(\Omega_p)\), for the sake of simplicity.

Set the \((\bm{\rho},\theta)\)-dependent  \((\mathrm{I}+2)^2\)-matrix
\[ 
\mathsf{A}(\bm{\rho},\theta) =\left[\begin{array}{ccccc}
D_1(\theta) &\cdots& D_{ 1\mathrm{I}}(\theta)  & a_{1,\mathrm{I}+1}(\rho_1,\theta) &a_{1,\mathrm{I}+2}(\rho_1,\theta)\\
\vdots&\ddots&\vdots &\vdots&\vdots\\
D_{\mathrm{I}1}(\theta) &\cdots &D_\mathrm{I}(\theta) &a_{\mathrm{I},\mathrm{I}+1}(\rho_\mathrm{I},\theta) &a_{ \mathrm{I},\mathrm{I}+2}(\rho_\mathrm{I},\theta) \\
a_{\mathrm{I}+1,1}(\rho_1,\theta) &\cdots &a_{\mathrm{I}+1,\mathrm{I}}(\rho_\mathrm{I},\theta) & k(\theta)  &a_{\mathrm{I}+1,\mathrm{I}+2}(\bm{\rho},\theta) \\
a_{\mathrm{I}+2,1}(\rho_1,\theta) &\cdots &a_{\mathrm{I}+2,\mathrm{I}}(\rho_\mathrm{I},\theta)  &a_{\mathrm{I}+2,\mathrm{I}+1}(\bm{\rho},\theta) & \sigma(\bm{\rho},\theta)
\end{array}\right] ,
\] 
where the leading coefficients are kept denoted according to the Fick, Fourier, Ohm laws, for reader's convenience.

The fuel cell problem, which its strong formulation is  stated in Section \ref{spemfc}, is equivalent to the following variational formulation.
\begin{definition} \label{dwt}
We say that the function \((\mathbf{u},p,\bm{\rho},\theta,\phi)\) is a weak solution to the fuel cell problem,
  if it satisfies the following variational formulations to
  \begin{itemize} 
\item  the momentum conservation (Beavers--Joseph--Saffman/Stokes--Darcy problem)
\begin{align} \label{motionwbj}
\int_{\Omega_\mathrm{f}}\mu(\theta)D\mathbf{u}:D\mathbf{v}  \dif{x}+
\int_{\Omega_\mathrm{f}}\lambda(\theta)\nabla\cdot\mathbf{u}\nabla\cdot\mathbf{v}  \dif{x} 
+ \int_{\Omega_\mathrm{p}}  \frac{K_g(p)}{\mu(\theta)}\nabla p\cdot\nabla v\dif{x} \nonumber \\
+\int_{\Gamma} \beta(\theta)\mathbf{u}\cdot\mathbf{v} \dif{s}+\int_{\Gamma }  p\mathbf{v}\cdot\mathbf{n} \dif{s}
-\int_{\Gamma }  \mathbf{u}\cdot\mathbf{n}v \dif{s} 
=R_\mathrm{specific}\int_{\Omega_\mathrm{f}} \rho \theta \nabla\cdot\mathbf{v} \dif{x}, 
\end{align}
 holds for all \((\mathbf{v},v)\in  \mathbf{V}(\Omega_f)\times H(\Omega_p)\).

\item the species conservation
\begin{align}  \label{wvfi} 
\int_{\Omega_\mathrm{f}}  \rho_i \mathbf{u}\cdot\nabla  v\dif{x} + 
\int_{\Omega} D_i(\theta )\nabla \rho_i \cdot\nabla v\dif{x} 
+  \sum_{\genfrac{}{}{0pt}{2}{j=1}{j\not= {i}}}^\mathrm{I}  \int_{\Omega_\mathrm{m}} D_{ij}( \theta )\nabla \rho_j \cdot\nabla v\dif{x}  \nonumber \\
+ \int_{\Omega} a_{i,\mathrm{I} +1 }(\rho_i,\theta )\nabla \theta\cdot\nabla v \dif{x} 
 +\int_{\Omega_\mathrm{p}}  a_{i, \mathrm{I}+2 } ( \rho_i,\theta )\nabla\phi \cdot\nabla v\dif{x} 
= 0, 
\end{align}
 holds for all  \(v\in V(\Omega)\)  and \(i=1,2,\cdots, \mathrm{I}\).
 
\item the energy conservation
\begin{align}\label{wvfi1} 
\int_{\Omega}  k(\theta)\nabla\theta \cdot\nabla v\dif{x} 
+ \int_{\Gamma_\mathrm{w}} h_c(\theta)\theta v \dif{s}  \nonumber \\
+\sum_{j=1}^{\mathrm{I} } 
\int_{\Omega} a_{\mathrm{I}+1 ,j}( \rho_j,\theta )\nabla \rho_j \cdot\nabla v\dif{x} +
\int_{\Omega_\mathrm{p}} a_{ \mathrm{I}+1 , \mathrm{I}+2}( \bm{\rho},\theta )\nabla\phi \cdot\nabla v\dif{x} \nonumber \\
 = \int_{\Gamma_\mathrm{w}}  h_c(\theta)\theta_e  v \dif{s} 
+\int_{\Omega_\mathrm{a}\cup\Omega_\mathrm{c}} \sigma( \bm{\rho},\theta ) |\nabla\phi|^2v\dif{x} ,
\end{align}
 holds for all  \(v\in  V(\Omega)\).
 
\item  the electricity conservation 
\begin{align}\label{wvfphi}
\int_{\Omega_\mathrm{p}}\sigma ( \bm{\rho},\theta )\nabla\phi\cdot\nabla w\dif{x}
 + \sum_{j=1}^{\mathrm{I}} \int_{\Omega_\mathrm{m}}  a_{\mathrm{I}+2,j}( \rho_j,\theta )\nabla \rho_j \cdot \nabla w \dif{x}  \nonumber \\ 
+  \int_{\Omega_\mathrm{m}} a_{\mathrm{I}+2, \mathrm{I}+1}( \bm{\rho},\theta )\nabla \theta \cdot \nabla w \dif{x}
+\int_{\Gamma_\mathrm{a}} j_a([\phi]) [w]\dif{s}\nonumber\\
=  \int_{\Gamma_\mathrm{c}}j_c(\phi_c- \phi_m-E_\mathrm{cell}) [w]\dif{s} , 
\end{align}
holds for all  \(w\in V(\Omega_p) \). 
\item and \(\rho\) obeying \eqref{massdensity}.
 \end{itemize}
\end{definition}
Hereafter, we use the notation \(\dif{s}\) for the surface element in
the integrals on the boundary  as well as any subpart of the  boundary \(\partial\Omega\).
Although in Section \ref{porousfluidbd}, the notation \([\cdot]\) was used for the jump of a quantity across  in the interface in the direction to the fluid media,
for the sake of clearness, 
in \eqref{wvfphi} it means
\([w]=w_\ell-w_m\), where the subscripts denote the restriction to \(\Omega_\ell\), \(\ell=\) a, c, or \(\Omega_\mathrm{m}\).

The equivalence between the strong and variational formulations use standard arguments \cite{vazquez}.
Indeed, the variational formulation \eqref{motionwbj} follows from the strong formulations \eqref{momentum}, \eqref{momentump} 
and \eqref{darcy1},  via the Green formula,
\begin{align*}
- \int_{\Omega_f} \tau :D\mathbf{v} \dif{x}
+ \langle \tau_T+\tau_N\mathbf{n},\mathbf{v} \rangle_\Gamma 
=&\int_{\Omega_f} p\nabla\cdot\mathbf{v} \dif{x} 
-\langle  p_\mathrm{f},\mathbf{v}\cdot\mathbf{n} \rangle_\Gamma
,\quad\forall \mathbf{v}\in \mathbf{V}(\Omega_f); \\ 
\int_{\Omega_\mathrm{p}}\frac{K_g(p)}{\mu (\theta)}\nabla p\cdot\nabla v\dif{x} =&
\int_{\Gamma}  \mathbf{u}_\mathrm{D}\cdot\mathbf{n} v \dif{s},
\quad \forall v\in H(\Omega_p), 
\end{align*}
by considering \eqref{cmf}, \eqref{cns} and \eqref{bj}.

The variational formulations \eqref{wvfi}, \eqref{wvfi1} and \eqref{wvfphi} follow from the respective strong formulations, namely,
from \eqref{species}, \eqref{speciesp} with boundary conditions \eqref{noflow}, \eqref{BVeq}-\eqref{etane};
from \eqref{heateqs},  \eqref{heatp}, \eqref{joule} with boundary conditions \eqref{newton}, \eqref{ttfp} and \eqref{bcsp};
and from \eqref{electric}  with boundary conditions \eqref{noflow}-\eqref{Ecell}.

\begin{remark}\label{meaningfull}
All terms are meaningful in the integral identities \eqref{motionwbj}-\eqref{wvfphi}. 
In particular, the Joule effect \(Q=\sigma |\nabla\phi|^2\)belonging to \( L^t(\Omega_\mathrm{a}\cup\Omega_\mathrm{c})\) is meaningful
for any \(t>1\) if \(n=2\) or for any \(t\geq 2n/(n+2)\) if \(n>2\).
\end{remark}

The set of hypothesis is as follows.
\begin{description}
\item[(H1)] The viscosities \(\mu\) and \(\lambda\) are  assumed to be Carath\'eodory functions from \(\Omega_\mathrm{f}\times\mathbb{R}\) into \(\mathbb{R}\) such that
\begin{align}\label{mu}
\exists \mu_\# ,  \mu^\# >0: &\  \mu_\# \leq\ \mu(x,e) \leq \mu^\# ; \\
\exists \lambda^\# >0: &\ |\lambda(x,e)| \leq  \lambda^\# ,
\label{nu3}
\end{align}
for a.e. \(x\in\Omega_\mathrm{f}\) and  for all \( e\in\mathbb{R}\).While \(K_g\) 
is assumed to be Carath\'eodory function from \(\Omega_\mathrm{p}\times\mathbb{R}\) into \(\mathbb{R}\) such that
\begin{equation}\label{defkl}
\exists K_l,b>0 : \ K_l\leq K_g(x,e)\leq K_l+b,
\end{equation}
for a.e. \(x\in\Omega_\mathrm{p}\) and  for all \( e\in\mathbb{R}\).

\item[(H2)]  
The matrix of coefficients \(\mathsf{A}\) has its components being  Carath\'eodory  functions from  \(\Omega\times\mathbb{R}^{\mathrm{I}+1}\) to \(\mathbb{R}\),
except the leading coefficients \(D_i\), \(k\)  that are Carath\'eodory  functions from  \(\Omega\times\mathbb{R}\) to \(\mathbb{R}\).
While  the leading coefficients \(D_i\), \(k\) and  \(\sigma\)   satisfy 
\begin{align}
\exists  D_i^\#, D_{i,\#},D_{i,p} >0:\ & D_{i,_\#} \leq  D_i(x,e)\leq D_i^\#,\quad \mbox{ for a.e. }  x\in \Omega_\mathrm{f} ; \label{Dif}\\
&  D_{i,_p} \leq  D_i(x,e)\leq D_i^\#,\quad \mbox{ for a.e. }  x\in \Omega_\mathrm{p} ; \label{Dip}\\
\exists k^\#, k_\# >0:\ & k_\#\leq k(x,e)\leq k^\#, \quad\mbox{ for a.e. } x\in\Omega; \label{km}\\
\exists\sigma^\#, \sigma_\#,\sigma_m >0:\ &\sigma_\#\leq \sigma(x,\mathbf{e})\leq \sigma^\#, \quad \mbox{ for a.e. }  x\in \Omega_\mathrm{a}\cup  \Omega_\mathrm{c};\label{smp}\\ 
&\sigma_m \leq \sigma(x,\mathbf{e})\leq \sigma^\#, \quad \mbox{ for a.e. }  x\in \Omega_\mathrm{m}\label{smm}
\end{align}
 for all  \(e\in\mathbb{R}\) and \(\mathbf{e}\in\mathbb{R}^{\mathrm{I}+1}\),   the remaining coefficients satisfy
\begin{align} 
\label{aI2j}\exists a_{\mathrm{I}+2,j}^\#>0:\quad& |a_{\mathrm{I}+2,j}(\cdot,\mathbf{e})|\leq a_{\mathrm{I}+2,j}^\#,\ \mbox{a.e. in } \Omega_\mathrm{m};\\
\label{aij}\exists a_{i,j}^\#>0:\quad& |a_{i,j}(\cdot,\mathbf{e})|\leq a_{i,j}^\#,\quad \mbox{a.e. in } \Omega,\ \forall\mathbf{e}\in\mathbb{R}^{\mathrm{I}+1} ,
\end{align}
for all \(i\in\{1,\cdots,\mathrm{I}+1\}\) and \(j\in\{1,\cdots,\mathrm{I}+2\}\) such that \(i\not=j\).
 Moreover,  we assume
\begin{align}\label{aa1}
&a_{i,\#}=\min\{D_{i,\#}, D_{i,p} \} -2^{\mathrm{I}}\frac{ (a^\#_{\mathrm{I}+1,i})^2}{k_\#}-2^{\mathrm{I}}\frac{ (a^\#_{\mathrm{I}+2,i})^2}{\sigma_m}
- 2^{\mathrm{I}+1}\sum_{\genfrac{}{}{0pt}{2}{j=1}{j\not= {i}}}^{\mathrm{I}} \frac{(a^\#_{j,i})^2}{D_{j,\#} } >0; \\
&a_{\mathrm{I}+1,\#} =k_\#-4\sum_{j=1}^{\mathrm{I}} \frac{(a^\#_{j,\mathrm{I}+1})^2}{D_{j,\#} } -2\frac{ (a^\#_{\mathrm{I}+2,\mathrm{I}+1})^2}{\sigma_m}>0; \\
&a_{\mathrm{I}+2,\#} =\min\{\sigma_\#,\sigma_m\}-2\frac{ (a^\#_{\mathrm{I}+1,\mathrm{I}+2})^2}{k_\#}-2\sum_{j=1}^{\mathrm{I}} \frac{(a^\#_{j,\mathrm{I}+2})^2}{D_{j,\#} }
>0,\label{aa3}
\end{align}
for each \(i= 1,\cdots,\mathrm{I}\). We  observe that these assumptions are required for the Legendre--Hadamard ellipticity condition.

\item[(H3)] The boundary coefficient \(\beta\)  is assumed to be a Carath\'eodory function from \(\Gamma\times\mathbb{R}\) into \(\mathbb{R}\). 
 Moreover, there exist \(\beta_\#,\beta^\#>0\) such that 
\begin{equation}\label{gamm}
\beta_\#\leq \beta(\cdot,e) \leq \beta^\# ,
\end{equation}
 a.e. in \(\Gamma\), and  for all \(e\in\mathbb{R}\).
 
\item[(H4)] The boundary coefficient \(h_c\)  is assumed to be a Carath\'eodory function from \(\Gamma_\mathrm{w}\times\mathbb{R}\) into \(\mathbb{R}\). 
 Moreover, there exist \(h_\#,h^\#>0\) such that 
\begin{equation}\label{hmm}
h_\#\leq h_c(\cdot,e) \leq h^\# ,
\end{equation}
 a.e. in \(\Gamma_\mathrm{w}\), and  for all \(e\in\mathbb{R}\).
 
\item[(H5)] The boundary functions \(j_\ell\), \(\ell=\) a, c, are assumed to be the odd continuous functions from \(\mathbb{R}\) into \(\mathbb{R}\), 
defined in \eqref{etapo}-\eqref{etane}.

\item[(H6)] There exists \( u_0\in H^{1}(\Omega_\mathrm{f})\) such that \(u_0=u_\mathrm{in}\) on \(\Gamma_\mathrm{in} \) and 
\(u_0=u_\mathrm{out}\) on \(\Gamma_\mathrm{out} \). Indeed, due to the characteristics of the problem, \(u_0\) has explicit expression
\[
u_0(x,y,z)= u_\mathrm{in} +(u_\mathrm{out}-u_\mathrm{in})y/L.
\]

\end{description}

\begin{remark}\label{choice}
The choice of \eqref{aa1}-\eqref{aa3} depends on the application of  the relation 
\((a_1+\cdots+a_N)^2\leq 2^{N-1}( a_1^2+a_2^2)+2^{N-2}a_3^2+\cdots +2a_N^2\), \(N>2\), in the inequality \eqref{cotatriple}.
\end{remark}

Using the fixed point argument, we establish the following result under the smallness on the data.
\begin{theorem}\label{tmain}
Let \(\Omega\) be a bounded multiregion domain   of \(\mathbb{R}^n\), \(n=2,3\).
Under the assumptions (H1)-(H6),
the fuel cell problem admits, at least, one solution according to Definition \ref{dwt} such that
\begin{itemize}
\item the velocity \(\mathbf{u}\in \mathbf{u}_0+\mathbf{V}(\Omega_f)\), with \(\mathbf{u}_0=u_0\mathbf{e}_y\);
\item the  pressure \(p\in H(\Omega_p)\);
\item the partial densities \( \bm{\rho} \in [V(\Omega)]^\mathrm{I} \);
\item the temperature \(\theta\in V(\Omega)\);
\item the potential \(\phi \in  E_{cell}\chi_{\Omega_\mathrm{c}}+  V_r(\Omega_p) \), for \(r>2\),
\end{itemize}
if provided by one of the smallness conditions \eqref{small1} or \eqref{small2}.
\end{theorem}
The existence of the weak solution to the fuel cell problem relies on the fixed point argument
\begin{align}\label{fpa}
(\pi,  \bm{\varrho} , \xi, \varphi, \Phi) & \in E:= H(\Omega_p)\times[H^1_{}(\Omega)]^{\mathrm{I}+1}  \times V(\Omega_p) 
\times L^t(\Omega_\mathrm{a}\cup\Omega_\mathrm{c})\nonumber\\
&\mapsto   (\mathbf{U},p) \in \mathbf{V}(\Omega_f)\times ( H(\Omega_p)\setminus\mathbb{R}) \nonumber \\
& \mapsto (\bm{\rho}, \theta,\phi_{cc}) \in   [V(\Omega)]^{\mathrm{I}+1} \times V(\Omega_p) \nonumber\\
&\mapsto (p, \bm{\rho},\theta,\phi_{cc}, |\nabla\phi|_{\Omega_\mathrm{a}\cup\Omega_\mathrm{c}}|^2)
\end{align}
where
\begin{itemize}
\item \((\mathbf{U},p) =(\mathbf{U} ,p) (\pi, \bm{\varrho},\xi)  \) stands for the auxiliary velocity-pressure pair given at Section \ref{auxup}; 
\item \( (\rho_1,\cdots,\rho_\mathrm{I}, \theta, \phi_{cc}) = (\bm{\rho}, \theta, \phi  )(\mathbf{w}, \bm{\varrho},\xi,\varphi, \Phi) \) 
stands for the auxiliary partial densities, temperature and potential
given at Section \ref{auxdi}, for \(t\geq 2n/(n+2)\) if \(n>2\) or  \(t>1\) if \(n=2\), with
  \(\mathbf{w}=\mathbf{u}(\pi, \bm{\varrho},\xi)\) being  the auxiliary velocity field given at Section \ref{auxup};

\item  \(\phi=\phi_{cc} + E_{cell}\chi_{\Omega_\mathrm{c}} \), 
with \(\chi_{\Omega_\mathrm{c}}\) denoting the characteristic function.
\end{itemize}

\section{Auxiliary results}
\label{tdt}

In this section, 
although our result is only valid for \(n=2,3\), we keep the space dimension \(n\) as general whenever possible.
Thus, the reader is able to be aware where the dimension is an obstacle and may reflect on it. 

We begin by naming some known constants (see, for instance, \cite{rossi}) that are used in this work.
\begin{definition}
We call by 
\begin{itemize}
\item \(S^*\) the continuity constant of the Sobolev embedding \(H^1(\Omega)\hookrightarrow L^{2^*}(\Omega)\), 
 i.e.  it obeys the Sobolev inequality
\begin{equation}\label{sobolev}
\|v\|_{2^*,\Omega} \leq S^* \|v\|_{1,2,\Omega},  \quad \forall v\in H^1(\Omega),
\end{equation}
with \(2^*=2n/(n-2)\) being the critical Sobolev exponent if \(n>2\).
If \(n=2\),  the Sobolev inequality holds for any  \(1\leq 2^*\leq \infty\). For the sake of simplicity, we also denote by \(2^*\)  any arbitrary real number
greater than one, if \(n=2\).
\item \(S_*\) the continuity constant of the trace embedding \(H^1(\Omega)\hookrightarrow L^{2_*}(\partial\Omega)\), 
i.e. it obeys the trace inequality
\begin{equation}\label{trace}
\|v\|_{2_*,\partial\Omega} \leq S_* \|v\|_{1,2,\Omega},  \quad \forall v\in H^1(\Omega),
\end{equation}
with \(2_*=2(n-1)/(n-2)\) being the critical trace exponent  if \(n>2\).
If \(n=2\),  we denote by \(2_*\) an arbitrary real number greater than one.
\end{itemize}
\end{definition}

\begin{remark}\label{rsob}
The Rellich--Kondrachov compact  embeddings \(W^{1,p}(\Omega)\hookrightarrow\hookrightarrow L^{p^*}(\Omega)\) and
 \(W^{1,p}(\Omega)\hookrightarrow\hookrightarrow L^{p_*}(\partial\Omega)\)
stand for any exponent between \(1\) and  the critical Sobolev exponent \(p^*\) and the critical trace exponent \(p_*\), respectively.
\end{remark}

The Poincar\'e constant  \(C_{\Omega}\) can have different forms,  \textit{i.e.}   it obeys one of the Poincar\'e inequalities
\begin{align}\label{poincarea}
\|v-\alpha\|_{2,\Omega} \leq C_{\Omega} \|\nabla v\|_{2,\Omega},& \quad \forall \alpha\in \mathbb{R},  v\in H^1(\Omega) ;\\
\|v\|_{2,\Omega} \leq C_{\Omega} \|\nabla v\|_{2,\Omega},  &\quad \forall v\in V(\Omega). \label{poincareb}
\end{align}

We recall that the following Korn inequality, where the constant is not explicitly determined because the proof relies on the
contradiction argument, is not useful for establishing quantitative estimates.
\begin{lemma}[Korn inequality]\label{korn}
Let \(\Omega\subset \mathbb{R}^n\), \(n\geq 2\), be a bounded Lipschitz domain, and let \(1<p<\infty\).
Then, there exists a constant \(C>0\) such that
\[
\|\nabla\mathbf{v}\|_{p,\Omega}\leq C\|D\mathbf{v}-\frac{1}{n}(\nabla\cdot\mathbf{v})\mathsf{I}\,\|_{p,\Omega}
\]
for all \(\mathbf{v}\in \mathbf{W}^{1,p}(\Omega)\).
\end{lemma}

Next,  the transport term is precised for some   exponent \(q\).
\begin{lemma}\label{lembu}
Let  \( \Omega \subset\mathbb{R}^n\) be a bounded Lipschitz domain. 
For each \( \mathbf{w}\in  \mathbf{L}^{q}(\Omega)\),   \(q=n>2\)  or \(q>n=2\),  the following functional is  well defined and  continuous: 
 \(
e\in H^1(\Omega)\mapsto\int_\Omega \mathbf{w}\cdot\nabla e v \dif{x},\)
for all \(v \in H^1(\Omega).\) In particular, the relation
\begin{equation}\label{advt}
\left|\int_{\Omega} \mathbf{w}\cdot\nabla e v \dif{x} \right|\leq \|\mathbf{w}\|_{q,\Omega} \|\nabla e\|_{2,\Omega} \| v\|_{2^*,\Omega} 
\end{equation}
holds for any \(e ,v\in H^1(\Omega)\). 
\end{lemma}
\begin{proof}
The wellposedness of the functional is consequence of the H\"older inequality,  for \(1/q+1/2^*=1/2\) \textit{i.e.}  \(2q/(q-2)=2^*\).
\end{proof}
Notice that  the Rellich--Kondrachov embedding \( H^1(\Omega)\hookrightarrow\hookrightarrow L^{p}(\Omega) \) is valid
with exponents \(q\), \(p\) and \(2\) such that (cf. Remark \ref{rsob})
\[
 \frac{1}{2^*}<\frac{1}{p}=\frac{1}{2}-\frac{1}{q} 
\Leftrightarrow q>n.
\]

\subsection{Auxiliary velocity-pressure pair}
\label{auxup}

For  \(\pi \in L^2(\Omega_\mathrm{p})\),
 \(\bm{\varrho}\in [ L^4(\Omega_\mathrm{f}) ]^{\mathrm{I}}\) and  \(\xi\in H^1(\Omega)\),
 we define the Dirichlet--BJS/Stokes--Darcy problem
\begin{align}
\int_{\Omega_\mathrm{f}}\mu(\xi)D\mathbf{U}:D\mathbf{v}  \dif{x}+
\int_{\Omega_\mathrm{f}}\lambda(\xi)\nabla\cdot\mathbf{U}\nabla\cdot\mathbf{v}  \dif{x}\nonumber \\
+\int_{\Gamma} \beta(\xi)\mathbf{U}_T\cdot\mathbf{v}_T \dif{s}+
 \int_{\Omega_\mathrm{p}}\frac{K_g(\pi)}{\mu (\xi)}\nabla p\cdot\nabla v\dif{x}
+ \int_\Gamma p\mathbf{v}\cdot \mathbf{n}\dif{s} -
 \int_{\Gamma} \mathbf{U}\cdot\mathbf{n} v\dif{s}\nonumber \\
= R_\mathrm{specific}\int_{\Omega_\mathrm{f}} \varrho\xi \nabla\cdot\mathbf{v} \dif{x}
-G(\xi, \mathbf{u}_0, \mathbf{v} )
,\ \forall (\mathbf{v},v) \in \mathbf{V}(\Omega_f)\times H(\Omega_p),\label{wvfup}
\end{align}
where
\begin{align}\label{defvarrho}
\varrho =&\sum_{i=1}^\mathrm{I}\varrho_i ,\\
G(\xi, \mathbf{z}, \mathbf{v} ) =&
\int_{\Omega_\mathrm{f}}\mu(\xi)D\mathbf{z}:D\mathbf{v}  \dif{x}
+\int_{\Omega_\mathrm{f}}\lambda(\xi)\nabla\cdot\mathbf{z}\nabla\cdot\mathbf{v}  \dif{x}\nonumber \\
&+\int_{\Gamma} \beta(\xi)\mathbf{z}_T\cdot\mathbf{v}_T \dif{s}
- \int_{\Gamma} \mathbf{z}\cdot\mathbf{n} v\dif{s}.
\end{align}

The existence of a unique weak solution \((\mathbf{U},p)=( \mathbf{U},p) (\pi, \bm{\varrho},\xi)\)
 to the variational equality \eqref{wvfup}  can be stated as follows.
\begin{proposition}[Auxiliary velocity-pressure pair]\label{proppxi}
Let \(\pi \in L^2(\Omega_\mathrm{p})\),
 \(\bm{\varrho}\in [ L^4(\Omega_\mathrm{f}) ]^{\mathrm{I}}\) and  \(\xi\in H^1(\Omega)\), \(n=2,3\).
 Under the assumptions (H1), (H3) and (H6),
the Dirichlet--BJS/Stokes--Darcy problem \eqref{wvfup} admits  a unique weak solution
 \((\mathbf{U},p)\in \mathbf{V}(\Omega_f)\times H(\Omega_p)\).
Moreover, if \(\mathbf{u}=\mathbf{U} +\mathbf{u}_0 \) the quantitative estimate
\begin{align}\label{cotaup}
\frac{n-1}{n}\mu_\# \|  D\mathbf{u} \|_{2,\Omega_\mathrm{f} }^2+ \beta_\#\|\mathbf{u}_T\|_{2,\Gamma} ^2 +
\frac{K_l}{\mu^\#} \|  \nabla p\|_{2,\Omega_\mathrm{p}}^2 \nonumber\\
\leq \left( \frac{R_\mathrm{specific}}{\sqrt{\mu_\#}}\|\varrho\|_{4,\Omega_\mathrm{f}}\|\xi\|_{4,\Omega_\mathrm{f}}
  +\sqrt{\mu^\#}\|D\mathbf{u}_0\|_{2,\Omega_\mathrm{f}} \right)^2 \nonumber\\
+ \lambda^\# \|\nabla\cdot\mathbf{u}_0\|_{2,\Omega_\mathrm{f}} ^2 +\max\{\beta^\#,\frac{\mu^\#}{K_l}\}\|\mathbf{u}_0\|_{2,\Gamma}^2
\end{align}
holds.  
\end{proposition}
\begin{proof}
 The existence of a unique weak solution \((\mathbf{U},p)\in \mathbf{V}(\Omega_f)\times (H(\Omega_p)/\mathbb{R})\)
  to the variational equality \eqref{wvfup} can be obtained by the Lax--Milgram lemma,
   due to  the assumptions (H1), (H3) and (H6).
Indeed, the uniqueness of \(p\) in \(H(\Omega_p)\) follows from the contradiction argument.
Assuming that \(p_1\) and \(p_2=p_1+c\), \(c\in\mathbb{R}\), satisfy \eqref{wvfup}, then subtracting the corresponding relations we obtain
\[
c\int_\Gamma \mathbf{v}\cdot \mathbf{n}\dif{s}=0,\ \forall \mathbf{v} \in \mathbf{V}(\Omega_f),
\]
which implies \(c=0\).

Instead to apply  Lemma \ref{korn}, we observe that
\begin{align}\label{coercive}
\frac{n-1}{n}\mu_\#\|D\mathbf{U}\|_{2,\Omega_\mathrm{f}}^2  
\leq\int_{\Omega_\mathrm{f}} \mu(\xi)\left(|D\mathbf{U}|^2-\frac{1}{n}|\nabla\cdot\mathbf{U}|^2\right) \dif{x} \nonumber\\
+\int_{\Omega_\mathrm{f}}  \left(\lambda (\xi)+\frac{\mu(\xi)}{n}\right) |\nabla\cdot\mathbf{U|^2}\dif{x},
\end{align}
using the fact that  \( (\nabla\cdot\mathbf{U})^2\leq|D\mathbf{U}|^2\) and \( n\lambda(\xi)+\mu(\xi)\geq 0\), 
for \(n\geq 2\).
The coercivity follows from the Poincar\'e inequality \eqref{poincarea} for \(p\) and from the inequality \eqref{coercive} for \(\mathbf{U}\).

The quantitative estimate \eqref{cotaup} follows from taking \((\mathbf{v},v)=(\mathbf{U},p)\) as a test function in \eqref{wvfup}, and next
taking the H\"older and Young inequalities into account,  applying  the assumptions  \eqref{mu}-\eqref{defkl}, and \eqref{gamm}, 
and using the inequality \eqref{coercive}.
\end{proof}

The continuous dependence is established as follows.
\begin{proposition}[Continuous dependence]\label{upm}
Suppose that the assumptions of Proposition \ref{proppxi} are fulfilled.
Let \(\{\pi_m\}\),  \(\{\bm{\varrho}_m\}\)   and \(\{\xi_m\}\) be sequences such that 
 \(\pi_m\rightarrow\pi\) in \(L^2(\Omega_\mathrm{p})\),
\(\bm{\varrho}_m\rightarrow\bm{\varrho} \) in \([L^4(\Omega)]^\mathrm{I}\),
and \(\xi_m\rightharpoonup \xi\) in \(H^1(\Omega)\), respectively.
If \((\mathbf{u}_m,p_m)= (\mathbf{U}+\mathbf{u}_0,p)(\pi_m, \bm{\varrho}_m, \xi_m)\) are the unique solutions to \eqref{wvfup}\(_m\),
 then \begin{align}\label{um}
\mathbf{U}_m \rightharpoonup \mathbf{U} \mbox{ in }\mathbf{V}(\Omega_f);\\
p_m\rightharpoonup p\mbox{ in }H(\Omega_p),\label{pm}
\end{align}
with \((\mathbf{u},p)= (\mathbf{U}+\mathbf{u}_0,p)(\pi, \bm{\varrho}, \xi)\) being the solution to \eqref{wvfup}.
\end{proposition}
\begin{proof}
Let  \(\{\pi_m\}\),  \(\{\bm{\varrho}_m\}\)   and \(\{\xi_m\}\) be sequences in the conditions of the proposition. 
The uniform estimate \eqref{cotaup} allows us to find a subsequence of \( \{(\mathbf{u}_m,p_m)\}\),  still denoted by \( \{(\mathbf{u}_m,p_m)\} \), such that 
the convergences \eqref{um}-\eqref{pm} hold.    It remains to prove that \((\mathbf{u},p)\) solves the   variational equality \eqref{wvfup}.

By appealing to the Rellich--Kondrashov compact embeddings \(H^1(\Omega)\hookrightarrow\hookrightarrow L^4(\Omega)\)  and
 \(H^1(\Omega)\hookrightarrow\hookrightarrow L^2(\Gamma )\), \(n=2,3\),  we have
\begin{align}
\xi_m\rightarrow \xi \mbox{ in } L^4(\Omega) &\mbox{ and a.e.  in } \Omega ; \label{xim2}\\
\xi_m\rightarrow \xi \mbox{ in } L^2(\Gamma  ) &\mbox{ and a.e.  on }\Gamma  .\label{ximg}
\end{align}
Applying the Krasnolselski to the Nemytskii operators \(K_g\) and \(\mu\), 
and the Lebesgue dominated convergence theorem, we obtain
\[
\frac{K_g(\pi_m)}{\mu(\xi_m)} \nabla v\rightarrow\frac{K_g(\pi)}{\mu(\xi)} \nabla v \mbox{ in } \mathbf{L}^2(\Omega_\mathrm{p} )  .\]
Analogously, for the coefficients \(\mu\), \(\lambda\) and \(\beta\).

Then, we pass to the limit the  variational equality \eqref{wvfup}\(_m\),  concluding that \((\mathbf{u},p)\) solves the   variational equality \eqref{wvfup}.
\end{proof}

\subsection{Auxiliary partial density-temperature-potential  triplet solution}
\label{auxdi}

In this section, we seek for the triplet solution \( (\bm{\rho},\theta, \phi) \).
Let  \(\mathbf{w}\in \mathbf{L}^{q}(\Omega_\mathrm{f})\)  be for 
\begin{equation}\label{defqn}
q\geq n>2 \mbox{ or } q>n=2.
\end{equation}
For    \(\bm{\varrho}\in [H^1(\Omega) ]^{\mathrm{I}}\),  \(\xi\in H^1(\Omega)\),  \(\varphi \in H^1(\Omega_\mathrm{p})\) and 
\(\Phi\in L^t(\Omega_\mathrm{a}\cup\Omega_\mathrm{c})\),  with
\begin{equation}\label{deft}
t\geq 2n/(n+2) \mbox{ if } n>2 \quad\mbox{  or  } \quad t>1 \mbox{ if } n=2,
\end{equation}
 we define the coupled problem
\begin{align}
\int_{\Omega_\mathrm{f}}  \rho_i \mathbf{w}\cdot\nabla  v\dif{x} + 
\int_{\Omega} D_i( \xi )\nabla \rho_i \cdot\nabla v\dif{x}  
+  \sum_{\genfrac{}{}{0pt}{2}{j=1}{j\not= {i}}}^\mathrm{I}  \int_{\Omega} D_{ij}( \varrho_j,\xi )\nabla \rho_j \cdot\nabla v\dif{x} \nonumber \\
+ \int_{\Omega} a_{i,\mathrm{I} +1 }( \varrho_i,\xi )\nabla \theta\cdot\nabla v \dif{x} 
 +\int_{\Omega_\mathrm{p}}  a_{i, \mathrm{I}+2 } ( \varrho_i,\xi )\nabla\phi \cdot\nabla v\dif{x} =0 ;  \label{aux-pdi} 
\end{align}\begin{align}
\int_{\Omega}  k(\xi)\nabla\theta \cdot\nabla v\dif{x} + \int_{\Gamma_\mathrm{w}} h_c(\xi)\theta v \dif{s} 
+\sum_{j=1}^{\mathrm{I} } \int_{\Omega} a_{\mathrm{I}+1 ,j}( \varrho_j,\xi)\nabla \rho_j \cdot\nabla v\dif{x}  \nonumber \\
+ \int_{\Omega_\mathrm{p}} a_{ \mathrm{I}+1 , \mathrm{I}+2}( \bm{\varrho},\xi)\nabla\phi\cdot\nabla v\dif{x} 
 = \int_{\Gamma_\mathrm{w}}  h_c(\xi)\theta_e  v \dif{s} 
+\int_{\Omega_\mathrm{a}\cup\Omega_\mathrm{c}} \sigma( \bm{\varrho},\xi) \Phi v\dif{x} ;\label{aux-tt} 
\end{align}
\begin{align} 
\int_{\Omega_\mathrm{p}}\sigma(\bm{\varrho},\xi) \nabla\phi\cdot\nabla w\dif{x}
 + \sum_{j=1}^{\mathrm{I}}\int_{\Omega_\mathrm{m}}  a_{\mathrm{I}+2,j} (\varrho_j,\xi)\nabla\rho_j \cdot \nabla w \dif{x}\nonumber\\
+\int_{\Omega_\mathrm{m}}  a_{\mathrm{I}+2,\mathrm{I}+1} (\bm{\varrho},\xi)\nabla\theta \cdot \nabla w \dif{x}  
+  \int_{\Gamma_\mathrm{a}} j_{a} ([\phi]) [w]\dif{s}
= \int_{\Gamma_\mathrm{c}} j_{c} ([\varphi]) [w]\dif{s},  \label{aux-phicc}\end{align}
 for every  \(i=1,2,\cdots, \mathrm{I}\), and for all  \(v\in V(\Omega)\) and \( w\in V(\Omega_p)\).

The existence of a unique weak  solution 
\(( \bm{\rho},\theta, \phi_{cc})=(\bm{\rho},  \theta, \phi) (\mathbf{w},  \bm{\varrho},\xi, \varphi, \Phi)\)
 to the variational equalities \eqref{aux-pdi}-\eqref{aux-phicc}  can be stated as follows.
\begin{proposition}[Auxiliary partial density-temperature-potential triplet]\label{proptt}
Let \(\mathbf{w}\in \mathbf{L}^q(\Omega_\mathrm{f})\),  \(q\geq n>2\)  or \(q>n=2\),  
be such that 
\begin{equation}\label{defw} \|\mathbf{w}\|_{q,\Omega_\mathrm{f}} <\min_{i} \frac{D_{i,\#}}{S^*}, \end{equation}
  \(\bm{\varrho}\in [ H^1(\Omega) ]^{\mathrm{I}}\), \(\xi\in H^1(\Omega)\),  \(\varphi \in H^1(\Omega_\mathrm{p})\)  and 
\(\Phi\in L^t(\Omega_\mathrm{a}\cup\Omega_\mathrm{c})\),  \(t\geq 2n/(n+2)\) if \(n>2\) or  \(t>1\) if \(n=2\),  be given.
Under the assumptions  (H2), (H4) and (H5),
the  variational problem \eqref{aux-pdi}-\eqref{aux-phicc}
admits a unique solution \((\bm{\rho},\theta, \phi_{cc})\in  [V(\Omega)]^{\mathrm{I}+1}\times V(\Omega_p) \).
Moreover, the quantitative estimate
\begin{align}\label{cotattphi}
\sum_{i=1}^\mathrm{I} \left(a_{i,\#} -S^*\|\mathbf{w}\|_{q,\Omega_\mathrm{f}} \right) \|\nabla\rho_i\|_{2,\Omega_\mathrm{f}}^2+
\sum_{i=1}^\mathrm{I} a_{i,\#} \|\nabla\rho_i\|_{2,\Omega_\mathrm{p}}^2
+a_{\mathrm{I}+2,\#} \|  \nabla\phi\|_{2,\Omega_p} ^2 \nonumber \\+
a_{\mathrm{I}+1,\#}\|\nabla\theta\|_{2,\Omega}^2+h_\#\|\theta\|_{2,\Gamma_\mathrm{w}}^2 \nonumber\\
\leq \frac{(S^*\sigma^\#)^2}{k_\#} \|\Phi\|_{t,\Omega_\mathrm{a}\cup\Omega_\mathrm{c}}^2
+\frac{ j_L^2 }{\min\{\sigma_\#,\sigma_m/2\}} 
+ h^\#\|\theta_e\|_{2,\Gamma_\mathrm{w}}^2
\end{align}
holds, for \(\phi=\phi_{cc}+E_\mathrm{cell}\chi_{\Omega_\mathrm{c}}\), \textit{i.e.} for a  solution being
 such that \(\phi=0\) on \(\Gamma_\mathrm{cc,a}\) and
\(\phi=E_\mathrm{cell}\) on \(\Gamma_\mathrm{cc,c}\).
\end{proposition}
\begin{proof}
Let \(\mathbf{w}\in \mathbf{L}^q(\Omega_\mathrm{f})\),  \(q=n>2\)  or \(q>n=2\),   
 \(\bm{\varrho}\in [ H^1(\Omega) ]^{\mathrm{I}}\), \(\xi\in H^1(\Omega)\) and 
\(\Phi\in L^t(\Omega_\mathrm{a}\cup\Omega_\mathrm{c})\), \(t= 2n/(n+2)\) if \(n>2\) or  \(t>1\) if \(n=2\), 
be fixed. 

 The existence of a unique weak solution  \((\bm{\rho},\theta, \phi_{cc} )\in  [V(\Omega)]^{\mathrm{I}+1} \times V(\Omega_p)\)
 to the variational equalities \eqref{aux-pdi}-\eqref{aux-phicc}  can be obtained by   the  Browder--Minty Theorem.
Indeed, the operator
\(L:  [V(\Omega)]^{\mathrm{I}+1}\times V(\Omega_p)  \rightarrow  \left( [V(\Omega)]^{\mathrm{I}+1}\times V(\Omega_p) \right)'\), defined by
\begin{align*}
\langle L(\mathbf{Y}),\mathbf{v}\rangle =\int_{\Omega}\mathsf{A}( \bm{\varrho},\xi) \nabla\mathbf{Y}\cdot\nabla \mathbf{v}\dif{x}
+ \int_{\Gamma_\mathrm{w}} h_c(\xi)\theta v \dif{s} \\
+  \int_{\Gamma_\mathrm{a}} j_{a} ([\phi]) [w]\dif{s}
+\sum_{i=1}^{\mathrm{I}}\int_{\Omega_\mathrm{f}}  \rho_i \mathbf{w}\cdot\nabla  v\dif{x} ,
\end{align*}
  where  \(\mathbf{Y}= (\bm{\rho},\theta, \phi)\) and \(\mathbf{v}=(v,\cdots, v,w)\),  
is hemicontinuous, strictly monotone and coercive (see \eqref{cotatriple}), if provided by \eqref{defw}.

Hereafter, 
 the symbol \(\langle\cdot,\cdot\rangle\)  denotes
the duality pairing \(\langle\cdot,\cdot\rangle_{X'\times X }\), with \(X\) being a Banach space and \(X'\) denoting the dual space of  \(X\).

Let us establish the quantitative estimate \eqref{cotattphi}.
We take \(v=\rho_i\), \(v=\theta\) and   \(w=\phi_{cc}\) as  test functions in \eqref{aux-pdi}, \eqref{aux-tt} and \eqref{aux-phicc}, respectively. 
Next, applying the limiting current bound \(j_L\),  the assumptions \eqref{Dif}-\eqref{aij},  
\eqref{hmm}, the H\"older and  Young inequalities,  also \eqref{advt},  and summing the  obtained expressions, we get
\begin{align}
\sum_{i=1}^\mathrm{I} \frac{D_{i,\#}}{2} \|\nabla\rho_i\|_{2,\Omega_\mathrm{f}}^2+
\sum_{i=1}^\mathrm{I} \frac{D_{i,p}}{2} \|\nabla\rho_i\|_{2,\Omega_\mathrm{p}}^2+
\frac{k_\#}{2}\|\nabla\theta\|_{2,\Omega}^2+ \frac{h_\#}{2}\|\theta\|_{2,\Gamma_\mathrm{w}}^2  \nonumber\\
+\sigma_\#\| \nabla\phi\|_{2,\Omega_\mathrm{a}\cup\Omega_\mathrm{c}}^2+\frac{\sigma_m}{2} \| \nabla\phi\|_{2,\Omega_\mathrm{m}}^2 \nonumber\\
\leq 
\sum_{i=1}^{\mathrm{I}}  \frac{1}{2 D_{i,\#}}\left(\sum_{\genfrac{}{}{0pt}{2}{j=1}{j\not= {i}}}^{\mathrm{I}}a^\#_{i,j}\|\nabla \rho_j\|_{2,\Omega}+
a^\#_{i,\mathrm{I}+1}\|\nabla \theta\|_{2,\Omega}+ a^\#_{i,\mathrm{I}+2} \|  \nabla\phi\|_{2,\Omega_\mathrm{p}}\right)^2 \nonumber\\
+\frac{1}{2k_\#} \left(\sum_{j=1}^{\mathrm{I}} a^\#_{\mathrm{I}+1,j} \|\nabla \rho_j\|_{2,\Omega}
+ a^\#_{\mathrm{I}+1,\mathrm{I}+2} \|  \nabla\phi\|_{2,\Omega_\mathrm{p}} \right)^2
+\frac{h^\#}{2}\|\theta_e\|_{2,\Gamma_\mathrm{w}}^2 \nonumber\\
+ S^*\|\mathbf{w}\|_{q,\Omega_\mathrm{f}} 
\sum_{i=1}^\mathrm{I}\|\nabla \rho_i\|_{2,\Omega_\mathrm{f}} \| \rho_i\|_{1,2,\Omega_\mathrm{f}} 
+\sigma^\# \|\Phi\|_{t,\Omega_\mathrm{a}\cup\Omega_\mathrm{c}}\|\theta\|_{t',\Omega_\mathrm{a}\cup\Omega_\mathrm{c}} \nonumber\\
+ \frac{1}{2\sigma_m}\left(\sum_{j=1}^{\mathrm{I}}  a_{\mathrm{I}+2,j} ^\# \|\nabla \rho_j\|_{2,\Omega_\mathrm{m}}
+a_{\mathrm{I}+2,\mathrm{I}+1} ^\# \|\nabla \theta \|_{2,\Omega_\mathrm{m}}\right)^2 
+j_L \|[\phi]\|_{1,\Gamma_\mathrm{c}}  .\label{cotatriple}
\end{align}
Then,  we consider the Sobolev imbedding \(V(\Omega)\hookrightarrow L^{t'}(\Omega)\) for \(t'= 2^*\), with the corresponding optimal Sobolev constant \(S^*\).
We emphasize that  the fundamental theorem of calculus may be applied to our special domain, and taking \(t'=4\) into account, explicit constants may be derived.
Instead using the trace-Poincar\'e inequality (namely, \(S_*\) and \(C_\Omega\)) on the last term in the right hand side, we apply
the fundamental theorem of calculus  to our special domain
\(\Omega_\mathrm{c}\) with \(v=0\) a.e. on \(\Gamma_\mathrm{c}\), \textit{i.e.}  at \(x=x_c+l_c\), and next  the Schwarz inequality,
\[
|v(x_c)|^2=\left|\int_{x_c}^{x_c+l_c} \partial_x v\dif{t}\right|^2 \leq l_c \int_{x_c}^{x_c+l_c}  |\partial_x v|^2\dif{t}.
\]
Hence,   we have
\begin{equation}
\int_{\Gamma_\mathrm{c}} |v|^2\dif{s} \leq l_c \int_{\Omega_\mathrm{c}}|\nabla v|^2 \dif{x}.
\end{equation}
Recall that the notation \(\dif{x}\) refers to the 2D \(\dif{x}\dif{y}\) and the 3D \(\dif{x}\dif{y}\dif{z}\).
Analogous inequality is valid for \(v=\varphi|_{\Omega_\mathrm{a}\cup \Omega_\mathrm{m}}\). Therefore, 
using Remark \ref{choice}  and \eqref{aa1}-\eqref{aa3}, the claimed quantitative estimate \eqref{cotattphi} arises.
\end{proof}

The continuous dependence is established as follows.
\begin{proposition}[Continuous dependence]\label{propttm}
Suppose that the assumptions of Proposition \ref{proptt} are fulfilled.
Let \(\{\mathbf{w}_m\}\),  \(\{\bm{\varrho}_m\}\),  \(\{\xi_m\}\), \(\{\varphi_m\}\)  and \(\{\Phi_m\}\)  be sequences such that 
\(\mathbf{w}_m\rightarrow \mathbf{w}\) in \(\mathbf{L}^q(\Omega_\mathrm{f})\), 
\(\bm{\varrho}_m\rightharpoonup \bm{\varrho} \) in \([H^1(\Omega)]^\mathrm{I}\),
 \(\xi_m\rightharpoonup\xi\) in \(H^1(\Omega)\),   \(\varphi_m\rightharpoonup\varphi\) in \(H^1(\Omega_\mathrm{p})\), 
and  \(\Phi_m\rightharpoonup\Phi\) in \( L^t(\Omega_\mathrm{a}\cup\Omega_\mathrm{c})\), respectively.
If \((\bm{\rho} _m,\theta_m,(\phi_{cc})_m)= (\bm{\rho} ,\theta,\phi)(\mathbf{w}_m, \bm{\varrho}_m, \xi_m,\varphi_m, \Phi_m)\) are the unique solutions to
 \eqref{aux-pdi}\(_m\)-\eqref{aux-phicc}\(_m\),
 then 
\begin{align}
\label{pdm}
\bm{\rho} _m \rightharpoonup \bm{\rho} \mbox{ in } [H^1(\Omega)]^\mathrm{I};\\
\theta_m\rightharpoonup \theta\mbox{ in }H^1(\Omega);\label{ttm}\\
\phi_m\rightharpoonup \phi\mbox{ in }H^1(\Omega_\mathrm{p}),\label{ppm}
\end{align}
with \((\bm{\rho} ,\theta, \phi_{cc})= (\bm{\rho} ,\theta, \phi)(\mathbf{w}, \bm{\varrho}, \xi,\varphi, \Phi)\) being the solution to \eqref{aux-pdi}-\eqref{aux-phicc}.
\end{proposition}
\begin{proof}
Let \(\{\mathbf{w}_m\}\),   \(\{\bm{\varrho}_m\}\),  \(\{\xi_m\}\), \(\{\varphi_m\}\)  and \(\{\Phi_m\}\)  be sequences in the conditions of the proposition,
and let \((\bm{\rho} _m,\theta_m, (\phi_{cc})_m)\) solve the corresponding  variational system \eqref{aux-pdi}\(_m\)-\eqref{aux-phicc}\(_m\).
Thanks to the estimate \eqref{cotattphi}, we can extract a (not relabeled) subsequence \(\{(\bm{\rho} _m,\theta_m, (\phi_{cc})_m)\}\) such that 
the convergences \eqref{pdm}-\eqref{ppm} hold.   

Let us prove that \((\bm{\rho} ,\theta, \phi_{cc})\) solves the corresponding variational equalities  \eqref{aux-pdi}-\eqref{aux-phicc}. 
By appealing to the Rellich--Kondrashov compact embeddings \(H^1(\Omega)\hookrightarrow\hookrightarrow L^2(\Omega)\),
\(H^1(\Omega)\hookrightarrow\hookrightarrow L^2(\Gamma_\mathrm{w} )\)  and
\(H^1(\Omega_\mathrm{p} )\hookrightarrow\hookrightarrow L^2(\Gamma_\mathrm{CL} )\)  for \(n\geq 2\),  we have
\begin{align}
 (\bm{\varrho}_m,\xi_m)\rightarrow (\bm{\varrho},\xi)\mbox{ in } [L^2(\Omega) ]^{\mathrm{I}+1} & \mbox{ and a.e.  in } \Omega ; \label{ae1}\\
\xi_m\rightarrow \xi\mbox{ in } L^2(\Gamma_\mathrm{w}  )& ;\label{ae2} \\
\varphi_m\rightarrow \varphi\mbox{ in } L^2(\Gamma_\mathrm{c}  ) & \mbox{ and a.e.  on } \Gamma_\mathrm{c};\label{ae3c}\\ 
\phi_m\rightarrow \phi\mbox{ in } L^2(\Gamma_\mathrm{a}  ) & \mbox{ and a.e.  on } \Gamma_\mathrm{a} .\label{ae3a}
\end{align}
Thanks to the continuity of the Nemytskii operators \(a_{l,j}\) for 
\(l,j=1,2,\cdots,\mathrm{I}+2\), where  \(a_{i,i}=D_i\),  \(k=a_{\mathrm{I}+1,\mathrm{I}+1}\),  and \(\sigma=a_{\mathrm{I}+2,\mathrm{I}+2}\),
using \eqref{ae1}  altogether the  Lebesgue dominated convergence theorem, we obtain 
\begin{align}\label{aljm}
a_{l,j} (\bm{\varrho}_m,\xi_m) \nabla v \rightarrow a_{l,j}(\bm{\varrho},\xi)\nabla v \mbox{ in } \mathbf{L}^2(\Omega ) ;\\
a_{\mathrm{I}+2,j} (\bm{\varrho}_m,\xi_m) \nabla w \rightarrow a_{\mathrm{I}+2,j}(\bm{\varrho},\xi)\nabla w \mbox{ in } \mathbf{L}^2(\Omega_\mathrm{p} ) ,\label{ym}
\end{align}
for all \(l=1,2\cdots,\mathrm{I}+1\).
Analogously  for the Nemytskii operator \(h_c\),  \(j_a \),  and \(j_c\), 
 using  \eqref{ae2}-\eqref{ae3a} and  the Lebesgue dominated convergence theorem,  we obtain
\begin{align}
h_c(\xi_m)v \rightarrow h_c(\xi)v  \mbox{ in }  L^2(\Gamma_\mathrm{w}); \nonumber\\
j_a([\phi_m])\rightarrow j_a([\phi])\mbox{ in } L^2(\Gamma_\mathrm{a}); \label{jma}\\
j_c([\varphi_m])\rightarrow j_c([\varphi])\mbox{ in } L^2(\Gamma_\mathrm{c}).\label{jmc}
\end{align}
Hence, we may pass to the limit the  variational  equalities \eqref{aux-pdi}\(_m\)-\eqref{aux-phicc}\(_m\).
Therefore, we conclude that \((\bm{\rho} ,\theta,\phi_{cc})\) solves the   variational equalities \eqref{aux-pdi}-\eqref{aux-phicc}.

By the uniqueness of limit, the weak limit of the initial sequence is the claimed solution.
\end{proof}

Next, we establish the convergence of the gradient a.e. in \(\Omega_\mathrm{a}\cup \Omega_\mathrm{c}\).
\begin{proposition}[Compactness]\label{aeconv}
If  \( \{(\bm{\varrho}_m,\xi_m, \varphi_m)\}_{m\in\mathbb{N}}\) is a sequence in \( [ H^1(\Omega_\mathrm{p} ) ]^{\mathrm{I}+2}\) such that weakly converges to 
\( (\bm{\varrho},\xi, \varphi )\) in \( [ H^1(\Omega_\mathrm{p} ) ]^{\mathrm{I}+2}\), and \((\phi_{cc})_m\) solves the corresponding  variational equality \eqref{aux-phicc}\(_m\), then
\(\nabla \phi_m\rightarrow\nabla\phi\) in \(L^2(\Omega_\mathrm{a}\cup \Omega_\mathrm{c})\), where
 \(\phi_{cc}=\phi-E_{cell}\chi_{\Omega_\mathrm{c}} \) solves the corresponding  variational equality \eqref{aux-phicc}.  In particular, 
\(\nabla \phi_m\rightarrow\nabla\phi\) almost everywhere, up to a subsequence, in \(\Omega_\mathrm{a}\cup \Omega_\mathrm{c}\).
\end{proposition}
\begin{proof}
Let \((\phi_{cc})_m\) and \(\phi_{cc}\) solve \eqref{aux-phicc}\(_m\) and \eqref{aux-phicc}, respectively.

Let us take \(w=(\phi_m-\phi)\chi_{\Omega_\mathrm{a}\cup \Omega_\mathrm{c}}\) as a test function in \eqref{aux-phicc}\(_m\) and  \eqref{aux-phicc}  
and subtracting the expressions, we obtain
\begin{align*}
\sigma_\#\int_{\Omega_\mathrm{a}\cup \Omega_\mathrm{c}}  |\nabla (\phi_m-\phi)|^2 \dif{x} \leq 
\int_{\Omega_\mathrm{a}\cup \Omega_\mathrm{c}}\sigma(\bm{\varrho}_m,\xi_m) |\nabla(\phi_m-\phi)|^2  \dif{x}\\
= I_1+I_2:=\int_{\Omega_\mathrm{a}\cup \Omega_\mathrm{c}} (\sigma(\bm{\varrho},\xi)-\sigma(\bm{\varrho}_m,\xi_m)) \nabla\phi\cdot\nabla(\phi_m-\phi) \dif{x}\\
+\int_{\Gamma_\mathrm{CL}} \left(j( \phi_m)-j( \phi)\right)(\phi_m-\phi)\dif{s} ,
\end{align*}
with \(j\) being defined by 
\begin{equation}\label{defj}
j(\phi)=\left\{
\begin{array}{ll}
j_a(\phi_a-\phi)& \mbox{ a.e. on } \Gamma_\mathrm{a}\\
-j_{c}(\varphi_c-\varphi)&\mbox{ a.e. on } \Gamma_\mathrm{c}
\end{array}\right.  
\end{equation}
where \(v_\ell\) denotes the trace of \(v|_{\Omega_\ell}\) on \(\Gamma_\ell\), \(\ell=\) a, c, while \(v\) stands for the trace of a function defined in \(\Omega_\mathrm{m}\).
Thus,  considering \eqref{ae1}-\eqref{ae3a}  and  \eqref{ym}, we conclude that each integral converges to zero,  and consequently
the proof of Proposition \ref{aeconv} is finished. 
\end{proof}

Finally,  some higher integrability can be obtained for the gradient (cf. \cite{lap2017} and the references therein).
\begin{proposition}[Regularity]\label{regular}
Let   \(\phi_{cc}\in V(\Omega_p)\) be the solution of the variational equality \eqref{aux-phicc}.
Then, \((\phi_{cc})|_{\Omega_\mathrm{a}\cup\Omega_\mathrm{c}}\)  belongs to the Sobolev space \(W^{1,2+\varepsilon}(\Omega_\mathrm{a}\cup\Omega_\mathrm{c})\),
for some  \(\varepsilon >0\) depending exclusively on the boundary, and the following quantitative estimate
\begin{equation}\label{cotajoule}
\|\nabla\phi\|_{r,\Omega_a\cup \Omega_c}\leq \frac{\sigma_\#M_r}{\sigma^\#
\left(\sigma^\#-M_r \sqrt{(\sigma^\#)^2-\sigma_\#^2}\right) } j_{L} |\Gamma_{CL}|:=R_3.
\end{equation}
holds.
Moreover, in the conditions of Proposition \ref{aeconv}, 
\(\sigma(\bm{\varrho}_m,\xi_m)|\nabla \phi_m|^2\rightarrow\sigma (\bm{\varrho},\xi)|\nabla\phi|^2\) in \(L^{1+\varepsilon/2}(\Omega_\mathrm{a}\cup \Omega_\mathrm{c})\).
\end{proposition}
\begin{proof}
Let \(\mathbf{Y} = (\bm{\varrho},\xi)\in [  H^1(\Omega_\mathrm{p}) ]^{\mathrm{I}+1}\)and \(\varphi \in  H^1(\Omega_\mathrm{p})\)  be fixed and let
  \(\phi_{cc}\in V(\Omega_p)\) be the solution of the variational equality \eqref{aux-phicc}.
Let us define the operator  \(A:V(\Omega_a\cup \Omega_c)\rightarrow \left(V(\Omega_a\cup \Omega_c)\right)'\) by
\[
\langle A(\mathbf{Y};\phi),w\rangle =\int_{\Omega_\mathrm{a}\cup \Omega_\mathrm{c}}\sigma(\mathbf{Y}) \nabla\phi\cdot\nabla w\dif{x}.
\]

By the uniqueness of solution,  the solution   \(\phi_{cc}\) verifies
\[
\langle A(\mathbf{Y};\phi),w\rangle = \int_{\Gamma_\mathrm{CL}}j(\phi)w\dif{s}, \ \forall w\in V(\Omega_a\cup \Omega_c),
\]
where \(j\) is defined in \eqref{defj}.

Denoting by \(A_r\) the restriction of \(A\) to \(V_r(\Omega_a\cup \Omega_c)\), and \(f=j(\phi)\in L^\infty(\Gamma_\mathrm{CL})\subset (V_r(\Omega_a\cup \Omega_c))'\),
  for any \(r>2\), the regularity established in the celebrated paper  by Gr\"oger and Rehberg \cite{groreh}
guarantees that there exists a \(r_0>2\) such that  \(A_r\) is bijective from \(V_r(\Omega_a\cup \Omega_c)\) onto \((V_r(\Omega_a\cup \Omega_c))'\),
for every \(r\in [2,r_0]\). The existence of \(r_0>2\) is determined by a class of the domain, which is formulated in \cite{gro89}. Indeed, the domains
\(\Omega_\mathrm{a}\) and \( \Omega_\mathrm{c}\) are regular in the sense formulated in  \cite{gro89},  for every \(r_0\geq 2\).
In particular,  it is proved that \(M_r< \sigma^\#/\sqrt{(\sigma^\#)^2-\sigma_\#^2}\), with 
\[
M_r :=\sup\{\|v\|_{1,r,\Omega_a\cup \Omega_c}: \ v\in V_r(\Omega_a\cup \Omega_c), \| Jv\|_{(V_r(\Omega_a\cup \Omega_c))'} \leq 1\},
\]
where
\[
\langle J\phi,w\rangle =\int_{\Omega_\mathrm{a}\cup \Omega_\mathrm{c}} \nabla\phi\cdot\nabla w\dif{x}.
\]
We remind that the dual space \(X'\) is equipped with
the usual induced norm \(\|f\|_{X'}=\sup \lbrace  \langle f,w\rangle, \ w\in X: \|w\|_X \leq 1 \rbrace\).

Moreover, the following estimate 
\[
\|A^{-1}\|_{\mathcal{L}((V_r(\Omega_a\cup \Omega_c))'; V_r(\Omega_a\cup \Omega_c))}\leq \frac{\sigma_\#M_r}{\sigma^\#
\left(\sigma^\#-M_r \sqrt{(\sigma^\#)^2-\sigma_\#^2}\right) }
\]
holds true,  where \(\|A^{-1}\|_{\mathcal{L}(Y;X)}=\sup\{\|A^{-1}y\|_X: \|y\|_Y\leq 1\}\).
Indeed, the following estimate
\[
\|\nabla\phi_1-\nabla\phi_2\|_{r,\Omega_a\cup \Omega_c}\leq \frac{\sigma_\#M_r}{\sigma^\#
\left(\sigma^\#-M_r \sqrt{(\sigma^\#)^2-\sigma_\#^2}\right) }\|j(\phi_1)-j(\phi_2)\|_{(V_r(\Omega_a\cup \Omega_c))'}
\]
holds true and plays an essential role. Indeed,  
the quantitative estimate  \eqref{cotajoule} holds true, due to the limiting current bound.

Considering the inequalities
\[ 
|a^2-b^2|^{r/2} \leq (a+b)^{r/2}|a-b|^{r/2}
\leq \frac{1}{2}(a+b)^{r} +  \frac{1}{2}|a-b|^{r}, \ \forall a,b>0,
\] 
we conclude that
\begin{align*}
\|\sigma(\mathbf{Y}_m)|\nabla\phi_m|^2-\sigma(\mathbf{Y})|\nabla\phi|^2\|_{r/2,\Omega_a\cup \Omega_c}  \\ \leq 
\|\sigma(\mathbf{Y}_m)\left(|\nabla\phi_m|^2-|\nabla\phi|^2\right)\|_{r/2,\Omega_a\cup \Omega_c}+
\|(\sigma(\mathbf{Y}_m)-\sigma(\mathbf{Y})) |\nabla\phi|^2\|_{r/2,\Omega_a\cup \Omega_c}\\
\leq \frac{\sigma_\#M_r}{\sigma^\#
\left(\sigma^\#-M_r \sqrt{(\sigma^\#)^2-\sigma_\#^2}\right) }\|j(\phi_m)-j(\phi)\|_{(V_r(\Omega_a\cup \Omega_c))'}.
\end{align*}
Therefore, the final claim is obtained,  by taking the strong convergences \eqref{ae3c}\eqref{ae3a} into account.
\end{proof}

By  considering Proposition \ref{regular},
we may define \(t = 1+\varepsilon/2\) that obeys \eqref{deft}.

\section{Fixed point argument(Proof of Theorem \ref{tmain})}
\label{fpthm}

Our aim is to apply the Tychonoff fixed point theorem to the operator \(\mathcal{T}\) defined in \eqref{fpa}.

The closed ball \(K\subset E= H(\Omega_p) \times  [V(\Omega)]^{\mathrm{I}+1}\times V(\Omega_p)
\times L^t(\Omega_\mathrm{a}\cup\Omega_\mathrm{c})\),  \(t>1\), defined as
\[
K= \lbrace (\pi, \bm{\varrho}, \xi,\varphi, \Phi): \ \|\pi\|\leq R_1, \left( \|\bm{\varrho}\|^2+\| \xi\|^2+\|\varphi\|^2\right)^{1/2} \leq R_2,
\|\Phi\|\leq R_3 \rbrace
\]
 is compact when the topological vector space is provided by the weak topology, or simply weakly compact, because \(E\) is reflexive.
The radius \(R_1\), \(R_2\) and \(R_3\) are the positive constants defined in \eqref{R1},  in the below cases (1) and (2), and in \eqref{cotajoule}, respectively.

The operator \(\mathcal{T}\) is well defined for \(n=2,3\),  due to Proposition \ref{proppxi}, 
considering that \(V(\Omega)\hookrightarrow L^4(\Omega)\) for \(n=2,3\), and due to Propositions
 \ref{proptt} and  \ref{regular}, 
taking \eqref{deft} and  \(\mathbf{w}=\mathbf{u}\in \mathbf{H}^{1}(\Omega_\mathrm{f}) \hookrightarrow \mathbf{L}^q(\Omega_\mathrm{f}) \) into  account, observing that
\(2n/(n-2)\geq  q\geq n=3,4\) or any \(q>n=2\).
Its continuity is due to Propositions \ref{upm} and \ref{propttm}-\ref{aeconv}, by providing  \( 6>q> n=3\) or
 any \(q> n=2\).

It remains to prove that \(\mathcal{T}\) maps \(K\) into itself. Let \( ( \pi, \bm{\varrho}, \xi,\varphi, \Phi)\in K\) be given, and 
\((p,\bm{\rho},\theta,\phi, |\nabla\phi|_{\Omega_\mathrm{a}\cup\Omega_\mathrm{c}}|^2)=\mathcal{T} (  \pi, \bm{\varrho}, \xi,\varphi, \Phi)\).
In particular there exists
 the auxiliary velocity field  \(\mathbf{u}=\mathbf{u}(\bm{\varrho}|_{\Omega_\mathrm{f}},\xi)\) being in accordance with  Section \ref{auxup}.

On the one hand, the estimate  \eqref{cotaup} may be rewritten 
\begin{align}
  \left( \|  D\mathbf{u} \|_{2,\Omega_\mathrm{f} }^2+ \|\mathbf{u}_T\|_{2,\Gamma} ^2\right)^{1/2} \leq aR_2^2+\frac{1}{\sqrt{\min\{\mu_\#/2,\beta_\#\}} }C_0; \\
\label{R1}
\|  \nabla p\|_{2,\Omega_\mathrm{p}} \leq  \sqrt{\frac{\mu^\#}{K_l} } \left(\frac{R_\mathrm{specific}}{\sqrt{\mu_\#}}R_2^2+C_0\right):= R_1 ,
\end{align}
where
\begin{align*}
a:=&\frac{1}{\sqrt{\min\{\mu_\#/2,\beta_\#\}} } \frac{R_\mathrm{specific}}{\sqrt{\mu_\#}} ; \\
C_0:= &\sqrt{\mu^\#}\|D\mathbf{u}_0\|_{2,\Omega_\mathrm{f}} 
+\sqrt{ \lambda^\#} \|\nabla\cdot\mathbf{u}_0\|_{2,\Omega_\mathrm{f}}  +\sqrt{\max\{\beta^\#,\frac{\mu^\#}{K_l}}\}\|\mathbf{u}_0\|_{2,\Gamma}.
\end{align*}

On the other hand, the estimate \eqref{cotattphi} yields
\begin{align}
\|\bm{\rho}\|_{2,\Omega}^2+\| \theta\|_{2,\Omega}^2 +\|\phi\|_{2,\Omega_\mathrm{p}}^2\leq  c/a_\# &\mbox{ if } b-aR_2^2\geq a_\#  ;\\
\|\bm{\rho}\|_{2,\Omega}^2+\| \theta\|_{2,\Omega}^2+\|\phi\|_{2,\Omega_\mathrm{p}}^2 \leq c/(b-aR_2^2) b&\mbox{ if } b-aR_2^2 < a_\#  ,
\end{align}
where
\begin{align*}
a_\#:=& \min_{j=1,\cdots,\mathrm{I}+2} a_{j,\#} ;\\
b:=&\min_{i=1,\cdots,\mathrm{I} } \frac{D_{i,\#}}{S^*}- \frac{1}{\sqrt{\min\{\mu_\#/2,\beta_\#\}} }C_0;\\
c:=&\frac{(S^*\sigma^\#)^2}{k_\#} R_3^2+\frac{ j_L^2 }{\min\{\sigma_\#,\sigma_m/2\}} + h^\#\|\theta_e\|_{2,\Gamma_\mathrm{w}}^2.
\end{align*}

The existence of \(R_2>0\) is guaranteed in both cases:
\begin{enumerate}
\item The case \( b-aR_2^2\geq a_\#\) means
\[c/a_\# \leq x:=R_2^2\leq (b-a_\#)/a, \]
which is true if provided by
\begin{equation}\label{small1}
 b\geq a_\# +ac/a_\#.
\end{equation}
\item The case \( b-aR_2^2 < a_\#\) means
\begin{align*}
(b-a_\#)/a<x:=R_2^2<b/a; &\\
b-\sqrt{\Delta} \leq 2ax \leq b+\sqrt{\Delta} & \mbox{ if } \Delta = b^2-4ac >0,
\end{align*}
which is true if provided by 
\begin{equation}\label{small2}
2\sqrt{ac} < b <2 a_\# +\sqrt{\Delta}.
\end{equation}
\end{enumerate}

Therefore, Theorem \ref{tmain} is completely proved.

\end{document}